\documentclass[a4paper]{article}
\usepackage{amsmath,amsfonts,amsthm, amssymb}

\newtheorem{theorem}{Theorem}[section] 
\newtheorem{lemma}[theorem]{Lemma}     
\newtheorem{corollary}[theorem]{Corollary}
\newtheorem{proposition}[theorem]{Proposition}

\newenvironment{proofof}[1]{\normalsize {\noindent \it Proof of #1}:}{{\hfill $\Box$}}
\newcommand{\R}{{\mathbb R}}

\newcommand{\N}{{\mathbb N}}
\newcommand{\C}{{\mathbb C}}
\newcommand{\CG}{{\mathbb C}G}
\newcommand{\Factors}{{\rm Fact}}
\newcommand{\FF}{{\rm F}}
\newcommand{\PP}{\mathcal{P}}
\newcommand{\D}{\mathcal{D}}

\newcommand{\LD}{{\rm LD}}
\newcommand{\RD}{{\rm RD}}
\newcommand{\lexle}{<_{\rm lex}}

\newcommand{\sgn}{{\rm Sign}}
\newcommand{\p}[1]{{\rm pre}[#1]}
\newcommand{\s}[1]{{\rm suf}[#1]}
\newcommand{\f}[1]{{\rm f}[#1]}
\renewcommand{\l}[1]{{\rm l}[#1]}
\newcommand{\DA}[1]{\mathrm{DA}(#1)}
\newenvironment{mylist}{\begin{list}{}{
\setlength{\parskip}{0mm}
\setlength{\topsep}{1mm}
\setlength{\parsep}{0mm}
\setlength{\itemsep}{0.5mm}
\setlength{\labelwidth}{7mm}
\setlength{\labelsep}{3mm}
\setlength{\itemindent}{0mm}
\setlength{\leftmargin}{12mm}
\setlength{\listparindent}{6mm}
}}{\end{list}}
\title{Rapid decay and Baum-Connes for large type Artin groups}
\author{Laura Ciobanu, Derek F. Holt and Sarah Rees}
\date{}
\begin{document}
\maketitle
\begin{abstract}
We prove that many Artin groups of large type satisfy the rapid decay
property, including all those of extra-large type. For many of these,
including all 3-generator groups of extra-large type, a result of
Lafforgue applies to show that the groups satisfy the Baum-Connes
conjecture without coefficients. 

Our proof of rapid decay combines elementary analysis with
combinatorial techniques, and relies on properties of geodesic
words in Artin groups of large type that were observed in \cite{HoltRees}
by two of the authors of this current article.\\


\end{abstract}

\section{Introduction}
\label{intro}
An Artin group $G$ is defined to be a finitely generated group with
presentation \[ \langle x_1,\ldots ,x_n \mid {}_{m_{ij}}(x_i,x_j)=
{}_{m_{ij}}(x_j,x_i)\quad\hbox{\rm for each}\quad i \neq j\quad\hbox{with}\quad m_{ij}< \infty \rangle, \]
where the integers $m_{ij}$ are the entries in a Coxeter matrix
(a symmetric $n \times n$ matrix over
$\N \cup \{\infty\}$, with $m_{ii}=1$, and $m_{ij} \geq 2$ for all $i \neq j$),
and,
for generators $x,y$ and $m \in \N$, the word ${}_m(x,y)$ is defined to be the
product of $m$ alternating $x$'s and $y$'s that starts with $x$. 
The associated Coxeter group is defined by adding the relations $x_i^2=1$
for all $i$.

A Coxeter matrix can be specified using an undirected labelled graph $\Gamma$
with $n$ vertices and an edge labelled $m_{ij}$ between distinct vertices $i$
and $j$ whenever $m_{ij} < \infty$, but no edge between $i$ and $j$ if
$m_{ij}$ is infinite. (We note that an alternative and widely used convention
deletes instead edges labelled 2 and leaves those with infinite labels.) We
denote the Artin group defined in this way by $G(\Gamma)$.

The Artin group
$G$ is said to be of {\em large type} if $m_{ij} \ge 3$ for all $i \neq j$,
and of {\em extra-large type} if $m_{ij} \ge 4$ for all $i \ne j$.
The concepts of large and extra-large type for Artin groups were introduced
by Appel and Schupp
\cite{AppelSchupp,Appel}, who used small cancellation arguments to prove 
solubility of the word and conjugacy problems, the embedding of parabolic
subgroups, the freeness of the subgroups $\langle x_i^2 \mid 1 \leq i \leq n
\rangle$, and more, for these groups.
An Artin group is said to be of {\em dihedral type} if 2-generated, of
{\em spherical type} if the associated Coxeter group is finite, or of {\em
right-angled type} if
$m_{ij} \in \{2,\infty\}$ for all $i,j$.

The purpose of this paper is to prove the following two results.

\begin{theorem}\label{main_rd}
Let $G=G(\Gamma)$ be an Artin group of large type for which $\Gamma$ has no
triangles with edge labels $3,3,m$ with $3 \le m < \infty$.
Then $G$ satisfies the rapid decay property.
In particular, all Artin groups of extra-large type satisfy this property.
\end{theorem}

\begin{corollary}\label{main_bc}
Let $G=G(\Gamma)$ be an Artin group satisfying the hypotheses of
Theorem~\ref{main_rd}.
Then $G$ satisfies the Baum-Connes conjecture without coefficients whenever one
of the following holds.
\begin{enumerate}
\item $G$ is 3-generated.
\item $\Gamma$ is triangle-free.
\item The edges of $\Gamma$ can be oriented in such a way that
the oriented graph contains no subgraph of either of the two types shown.

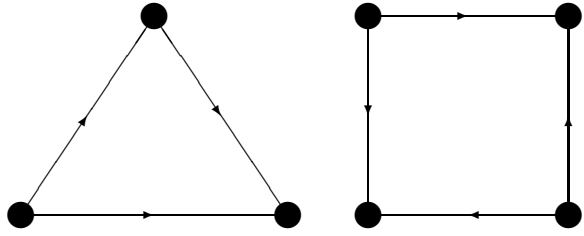
\begin{picture}(200,110)(0,-10)
\put(0,0){\circle*{10}}
\put(0,0){\line(1,0){100}}
\put(50,0){\vector(1,0){0}}
\put(0,0){\line(2,3){50}}
\put(25,37.5){\vector(2,3){0}}
\put(100,0){\circle*{10}}
\put(100,0){\line(-2,3){50}}
\put(75,37.5){\vector(2,-3){0}}
\put(50,75){\circle*{10}}
\put(130,0){\circle*{10}}
\put(130,0){\line(1,0){75}}
\put(167.5,0){\vector(-1,0){0}}
\put(130,0){\line(0,1){75}}
\put(130,37.5){\vector(0,-1){0}}
\put(205,0){\circle*{10}}
\put(205,0){\line(0,1){75}}
\put(205,37.5){\vector(0,1){0}}
\put(130,75){\circle*{10}}
\put(130,75){\line(1,0){75}}
\put(167.5,75){\vector(1,0){0}}
\put(205,75){\circle*{10}}
\end{picture}
\end{enumerate}

In particular, any large type Artin group $G(\Gamma)$  for which $\Gamma$
is triangle-free satisfies the Baum-Connes conjecture without coefficients.
\end{corollary}

Following Jolissaint \cite{J90}, we say that a finitely generated group $G$
has the rapid decay property (RD) if
the operator norm $||.||_*$ for the group algebra $\CG$ is bounded by a constant multiple
of the Sobolev norm $||.||_{s,r,\ell}$, a norm that is a variant of the $l^2$ norm
weighted by a length function for the group.

More precisely, rapid decay holds for $G$ if there are constants $C,r$ and a
length function $\ell$ on $G$ such that
for any $\phi,\psi \in \C{G}$,
\[ ||\phi||_* := \sup_{\psi \in \CG}\frac{|| \phi*\psi||_2}{||\psi||_2} \leq C||\phi||_{2,r,\ell}.\] 
Here,
$\phi*\psi$ denotes the convolution of $\phi$ and $\psi$,
$||.||_2$ is the
standard $l^2$ norm, and $ ||.||_{2,r,\ell}$ the Sobolev norm of order $r$, with respect to the length function $\ell$, that is
\begin{eqnarray*}
\phi*\psi(g)&=&\sum_{h \in G}\phi(h)\psi(h^{-1}g),\\
||\psi||_2 &=& \sqrt{\sum_{g \in G}|\psi(g)|^2}\\
||\phi||_{2,r,\ell}&=& \sqrt{ \sum_{g \in G} |\phi(g)|^2 (1 + \ell(g))^{2r}}.
\end{eqnarray*}
We call a function $\ell:G \rightarrow \R$ a length function for $G$ if it satisfies
\[ \ell(1_G)=0,\quad\ell(g^{-1})=\ell(g),\quad
\ell(gh) \leq \ell(g)+ \ell(h), \quad \forall g,h \in G. \]
An account of the rapid decay property is given in \cite{CR05}.

The main interest in the rapid decay property for a group $G$ stems
from Lafforgue's result \cite{Lafforgue} that
the combination
of rapid decay with an
appropriate action of $G$
implies that the Baum-Connes conjecture without coefficients holds for $G$. 
That conjecture relates the $K$-theory of
the reduced $C^*$-algebra
$C^*_r(G)$ to the $K^G$-homology of the classifying space $\underline{E}G$
for proper $G$-actions, claiming
that the assembly maps
\[ \mu^G_i:RK^G_i(\underline{E}G)\rightarrow K_i(C^*_r(G)),\quad
i=0,1,\]
are isomorphisms.
This in turn implies the Novikov conjecture and (when $G$ is torsion-free)
the Kaplansky-Kadison idempotent conjecture; see \cite{V02}.

The conjecture as stated above is commonly referred to as the Baum-Connes
conjecture
`without coefficients' in order to distinguish it from a more general
conjecture, with coefficients from a $C^*$-algebra admitting an action of $G$,
to which counter-examples were claimed in \cite{HigsonLafforgueSkandalis}.

Braid groups (Artin groups of type $A_n$) are shown to satisfy rapid
decay in \cite{BehrstockMinsky}, where the 
result is proved for mapping class groups; braid groups are proved to satisfy Baum-Connes in
\cite[Corollary 14]{Schick}. We note that since Artin groups of type $B_n$
and $D_n$ are split extensions of free groups by braid groups
\cite{CrispParis}, it follows immediately from \cite[Corollary 14]{Schick}
that those spherical type
Artin groups also satisfy Baum-Connes (but we do not know if the result holds
for the remaining spherical type groups).
Right-angled Artin groups are proved to satisfy both rapid decay and Baum-Connes through their
action on CAT(0) cube complexes \cite{CR05}; alternative proofs of their rapid decay follow
from their small cancellation properties \cite{Ong}, or the fact that they
are graph products of infinite cyclic groups \cite{CHR}. Dihedral Artin groups
are easily proved to satisfy rapid decay since they 
are virtually direct products of free and cyclic groups. For the same
reason they
satisfy the Haagerup property (see \cite{CherixEtAl}),
which implies Baum-Connes \cite{HigsonKasparov}. We are not aware
of any other classes of Artin groups known to satisfy either rapid decay or
Baum-Connes prior to our results. 

We shall call the hypothesis in Theorem~\ref{main_rd} that there no
triangles in $\Gamma$ with edge labels $3,3,m$ the {\em $(3,3,m)$-hypothesis}.
We shall prove rapid decay for Artin groups of large type that
satisfy this hypothesis, using
as our length function the word length metric over the standard generating set;
that is, for each element $g$ we shall define $\ell(g)$ to be the length of
the shortest word over the standard generators that represents $g$.

Our proof makes critical use of the results and techniques developed in
the earlier paper \cite{HoltRees} by the second and third authors, in which
the sets of geodesics and shortlex minimal geodesics of Artin groups of
large type in their standard presentations are studied. It is proved
there that those sets of geodesics are regular, and that the groups are
shortlex automatic. Furthermore, a method was developed for rewriting
arbitary words in the group generators to their shortlex normal forms.
In Sections~\ref{dihedral} and~\ref{largetype} of this paper,
we shall recall some of these results and techniques, and develop them further,
proving that certain conditions hold for factorisations of geodesics.
In particular, our proof of Theorem  \ref{main_rd} depends essentially on
Proposition \ref{onetail}, which turns out to be false in all Artin groups of
large type that do not satisfy the $(3,3,m)$-hypothesis, so we are currently
unable to dispense with this hypothesis.

We shall prove that certain conditions hold for the factorisation of
geodesics, and derive rapid decay in groups for which these conditions hold.
For many of those groups we can build on work of \cite{BradyMcCammond} to
construct an action of
the group that allows the application of Lafforgue's result to deduce Baum-Connes.

Our strategy to prove rapid decay by examining the
factorisation of geodesics was also used by
Jolissaint \cite{J90} and de la Harpe \cite{H88} to prove rapid decay
for word hyperbolic groups, extending Haagerup's
proof for free groups \cite{Haagerup}; Drutu and Sapir used similar
techniques to prove rapid decay for a group
hyperbolic relative to a family of subgroups with the property
\cite{DS05}.
Other authors have used more geometric techniques \cite{Chatterji,CR05}.

In Section~\ref{RDtoBC} we explain how Corollary~\ref{main_bc} can be deduced
from Theorem~\ref{main_rd}.
Then Section~\ref{reformulation} shows how rapid
decay can be deduced from a pair of combinatorial conditions D1 and D2
relating to the factorisation of geodesics in the group. Section~\ref{notation}
introduces the notation we need for the remainder of the article. Section~\ref{dihedral} examines dihedral Artin groups, first recalling from \cite{HoltRees}
some technical results on the structure of geodesics, and then using these
to deduce the conditions D1 and D2 for these groups. Section~\ref{largetype}
extends the results of Section~\ref{dihedral} to deduce the same conditions
for the large type groups considered in Theorem~\ref{main_rd}.

\section{Deducing Baum-Connes from rapid decay}
\label{RDtoBC}
Any discrete group that
acts continuously,
isometrically, properly and co-compactly on a CAT(0) metric space is in
the class ${\mathcal C}'$ defined by Lafforgue in \cite{Lafforgue};
hence by \cite[Corollary 0.4]{Lafforgue},
for such a group the Baum-Connes conjecture is a consequence of rapid decay.

We shall deduce Corollary~\ref{main_bc} from Theorem~\ref{main_rd} using
results of \cite{BradyMcCammond} to construct an appropriate CAT(0)
action for a large class of Artin groups of large type.

A CAT(0) space is defined to be a metric space in which distances across a triangle
with geodesic sides of lengths $p,q,r$ are always bounded above by distances
between equivalent points in a Euclidean triangle with the same side lengths.
A space is locally-CAT(0) or equivalently non-positively curved if
any point has a neighbourhood that is CAT(0).
By the Cartan-Hadamard theorem of \cite{Ballmann2}, a simply connected locally-CAT(0) space is actually CAT(0).

Brady and McCammond define non-positively
curved presentation complexes for non-standard presentations of
3-generator Artin groups of large
type (and some others) in \cite{BradyMcCammond}. We shall see that
the actions of the groups on the universal covers of these
complexes satisfy the conditions we need.

Theorems 4, 6  and 7 of that paper consider the three cases of our theorem.
In each of the three cases we define a complex $K_\Gamma$, the presentation
complex for a presentation $I_\Gamma$ of the Artin group $G(\Gamma)$
in which all relators have length 3, 
formed from the standard presentation by the addition of some generators.
The presentation complex $K_\Gamma$ has
a single 0-cell $v_0$, a 1-cell for each generator of $I_\Gamma$, and a
triangular 2-cell for each relator of $I_\Gamma$.
We define $L_\Gamma$ to be the link of $v_0$;
this is the intersection with the surrounding complex
of a small sphere centred on $v_0$,
and can be viewed as a graph 
to which each edge of $K$ contributes two vertices, and each corner of a
2-cell contributes an edge.
$\tilde{K_\Gamma}$ is the universal cover of $K_\Gamma$;
this is simply connected,
and also has $L_\Gamma$ as the link of each 0-cell.

In each of the three cases of our theorem, it is proved
in \cite{BradyMcCammond} that some specification of
edge lengths of the triangles which are the 2-cells of $K_\Gamma$
extends to a piecewise Euclidean metric on $K_\Gamma$;
in the first and third cases of the theorem all triangles are made
equilateral with side length 1, and in the second case all triangles are
made isosceles and right-angled with their shorter sides of length 1.
The piecewise Euclidean metric induces a metric on the link $L_\Gamma$ of $v_0$,
where the length of each edge is equal to the angle at the corner through
$v_0$ of the corresponding 2-cell.
The metric can also be lifted to the cover
$\tilde{K_\Gamma}$.

An action of $G$ by isometries on $\tilde{K}_\Gamma$ is inherited from
the left regular action of $G$ on its Cayley graph (the 1-skeleton
of $\tilde{K}_\Gamma$), in which the vertex $h$
is mapped by $g$ to the vertex $gh$.
The action is free, so clearly it is both proper
and continuous. The quotient, $K_\Gamma$, of $\tilde{K_\Gamma}$ by $G$ is certainly compact.
In order to apply Lafforgue's result,
we need simply to verify that $\tilde{K_\Gamma}$ equipped with this metric
is CAT(0).

That $K_\Gamma$ is non-positively curved (locally-CAT(0)) is proved in \cite{BradyMcCammond}. In each of the three cases the length of a closed path
in the link of a vertex is seen to be at least $2\pi$; hence
\cite[Theorem 15]{Ballmann} applies to show that $K_\Gamma$ is non-positively curved.
Exactly the same
argument can be applied to the link of a vertex in $\tilde{K_\Gamma}$ to deduce
that the simply connected cover $\tilde{K_\Gamma}$ is also non-positively curved, and hence by the Cartan-Hadamard
theorem of \cite{Ballmann2}, $\tilde{K_\Gamma}$ is CAT(0). Baum-Connes
in each of these cases now follows by \cite[Corollary 0.4]{Lafforgue}.

\section{Reformulation of the rapid decay condition}
\label{reformulation}
Let $\ell$ be a length function on a group $G$.
Given $k\in \N$, we define $C_k$
to be the set of
elements of $G$ with $\ell(g)=k$.
We write $\chi_k$ for the characteristic function on $C_k$, 
and for $\phi \in \CG$, we write $\phi_k$ for the pointwise product
$\phi.\chi_k$.
(Generally, in this article we use a subscript $k$ to indicate that a function
has support on $C_k$.)

It is proved by Jolissaint \cite[Proposition 1.2.6]{J90} that rapid decay for $G$ is equivalent to the following condition $(*)$:
\begin{eqnarray*}
 \forall \phi,\psi \in \CG,k,l,m \in \N,\\
|k-l|\leq m \leq k + l,&\Rightarrow&
||(\phi_k * \psi_l)_m ||_2 \leq  ||\phi_k||_{2,r,\ell} || \psi_l||_2,\\
\hbox{\rm otherwise}&&
||(\phi_k * \psi_l)_m ||_2 =0\end{eqnarray*}
We observe that it follows from the properties of
a  length function 
that the norm $||(\phi_k * \psi_l)_m ||_2$ is zero for $m$ outside
the range $[|k-l|,k+l]$. Hence we shall verify rapid decay by verifying
the following condition $(**)$:

there exists a polynomial $P(x)$ such that:
\begin{eqnarray*}
 \forall \phi,\psi \in \CG,k,l,m \in \N,\\
|k-l|\leq m \leq k + l,&\Rightarrow&
||(\phi_k * \psi_l)_m ||_2 \leq   P(k)||\phi_k||_2 || \psi_l||_2.
\end{eqnarray*}
We can clearly assume that the coefficients of such a polynomial $P$ are
non-negative, and hence that $P$ is an increasing function of $x$, and we
shall assume throughout this paper that all polynomials that arise have this
property.

\section{Geodesic factorisation}
For $g \in G$ we call a decomposition of $g$ as a product $g_1\cdots g_k$
a {\em geodesic factorisation} of $g$ if 
$\sum \ell(g_i)=\ell(g)$, and call the elements $g_i$ {\em divisors} of $g$.
In particular $g_1$ is called a {\em left divisor} and $g_k$ is called a
{\em right divisor}.
Given $g\in C_{k+l}$
we define \begin{eqnarray*}
\Factors_{k,l}(g)&:=& \{ (g_1,g_2): g_1 \in C_k,g_2 \in C_l: g_1g_2 =_G g \} \\
\FF_{k,l}&:= & \sup_{g \in C_{k+l}} |\Factors_{k,l}(g)|\\
\end{eqnarray*}
Let $\PP$ be a subset of $G^2$ and, for $g \in G$, let
$\PP(g) = \{(g_1,g_2) \in \PP : g_1g_2=g\}$.
We shall refer to the factorisations of $g$ in $\PP(g)$ as the
{\em permissible} factorisations of $g$.
We define
\begin{eqnarray*}
\Factors_{\PP,k,l}(g)&:=&
\{ (g_1,g_2): g_1 \in C_k,g_2 \in C_l, (g_1,g_2) \in \PP(g) \} \\
\FF_{\PP,k,l}&:= & \sup_{g \in C_{k+l}} |\Factors_{\PP,k,l}(g)|\\
\end{eqnarray*}
We shall verify the condition $(**)$ above for Artin
groups satisfying the hypotheses of Theorem~\ref{main_rd},
and hence verify rapid decay,
by finding a suitable set $\PP$ of permissible factorisations
for which the size of both $\Factors_{\PP,k,l}(g)$ and a related
set are polynomially bounded.

We define two conditions D1, D2 that we shall require to hold on a 
suitable subset $\PP$ of $G^2$.
\begin{description}
\item[D1] $\FF_{\PP,k,l}$ is bounded above by $P_1(\min(k,l))$ for some
polynomial $P_1(x)$.

\item[D2] 
For each $g \in G$, each $k,l \geq 0$,
there is a subset $S(g,k,l)$ of $G^3$ 
as follows.

For each decomposition of $g$ as a product
$g_1g_2$ with $g_1 \in C_k$, $g_2 \in C_l$,
$S(g,k,l)$ contains a triple $(f_1,\hat{g},f_2)$,
for which
$g=f_1\hat{g}f_2$, and $\hat{g}=h_1h_2$, where
$(f_1,h_1) \in \Factors_{\PP,k-p_1,p_1}(g_1)$ and
$(h_2,f_2) \in \Factors_{\PP,p_2,l-p_2}(g_2)$,
for some $p_1,p_2 \le K\min(k,l)$, for some global constant $K$.

Furthermore, there are polynomials $P_2(x), P_3(x)$ such that
\begin{description}
\item[(a)] for all $g,k,l$, $|S(g,k,l)| \leq P_2(\min(k,l))$,
\item[(b)] $|T(k,l)| \le P_3(\min(k,l))$, where
$T(k,l) := \{ \hat{g}: \exists g,\, (f_1,\hat{g},f_2) \in S(g,k,l) \}$.
\end{description}
\end{description}
Figure~\ref{fig1} illustrates the condition D2.

\setlength{\unitlength}{0.65pt}
\begin{figure}
\begin{center}
{\large
\begin{picture}(500,260)(-250,-110)
\put(-250,0){\circle*{10}}
\put(0,-40){\circle*{10}}
\put(250,0){\circle*{10}}
\qbezier(-250,0)(0,50)(250,0) \put(0,30){$g$}
\qbezier(-250,0)(-125,0)(0,-40) \put(-140,-20){$g_1$}
\qbezier(250,0)(125,0)(0,-40) \put(125,-20){$g_2$}
\put(-100,-60){\circle*{10}}
\put(100,-60){\circle*{10}}
\qbezier(-250,0)(-175,-25)(-100,-60) \put(-160,-50){$f_1$}
\qbezier(250,0)(175,-25)(100,-60) \put(150,-50){$f_2$}
\qbezier(-100,-60)(0,-25)(100,-60) \put(0,-60){$\hat{g}$}
\put(-80,-45){$h_1$} \put(70,-45){$h_2$}
\end{picture}
}
\caption{\label{fig1} Condition D2}
\end{center}
\end{figure}
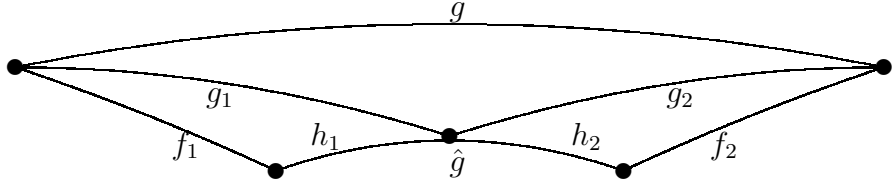

\begin{theorem}
\label{D1D2_RD}
Let $G$ be a group and $\PP$ a subset of $G^2$ for which the conditions D1 and D2 hold.
Then $G$ satisfies rapid decay with respect to the word length metric.
\end{theorem} 

Before launching into the proof of Theorem~\ref{D1D2_RD} we prove a few
technical results.  
Suppose first that $\phi_{k}$ is a function with support on $C_{k}$.
Then for any $p\ge 0$, we can define functions
$\phi^{(p)}_{\PP,k-p}$ and ${}^{(p)}\phi_{\PP,k-p}$ with support on $C_{k-p}$ by
\begin{eqnarray*}
\phi^{(p)}_{\PP,k-p}(g) &=& \sqrt{\sum_{h \in C_{p}, (g,h) \in \PP} | \phi_k(gh)|^2}, \quad\hbox{\rm if}\quad g \in C_{k-p}\\
&=& 0\quad\hbox{\rm otherwise}\\
{}^{(p)}\phi_{\PP, k-p}(g) &=& \sqrt{\sum_{h \in C_{p}, (h,g) \in \PP} | \phi_k(hg)|^2}, \quad\hbox{\rm if}\quad g \in C_{k-p}\\
&=& 0\quad\hbox{\rm otherwise.}
\end{eqnarray*}
%
\begin{lemma}
\label{relate_l2_norms}
\begin{eqnarray*}
 ||\phi^{(p)}_{\PP,k-p}||_2^2 &\leq& \FF_{\PP,k-p,p}||\phi_{k}||_2^2 \\
 ||{}^{(p)}\phi_{\PP,k-p}||_2^2 &\leq& \FF_{\PP,p,k-p}||\phi_{k}||_2^2 \end{eqnarray*}
\end{lemma}
\begin{proof}
\begin{eqnarray*}
||\phi^{(p)}_{\PP,k-p}||_2^2 &=& \sum_{g \in C_{k-p}}\sum_{h \in C_{p}, (g,h)
\in \PP}|\phi_k(gh)|^2 \\ & \leq & \FF_{\PP,k-p,p} \sum_{g_1 \in C_{k}}
|\phi_k(g_1)|^2\\ &=& \FF_{\PP,k-p,p} ||\phi_k||_2^2\\ \end{eqnarray*}
The second inequality follows similarly.
\end{proof}
We shall make frequent use of the following, which
is an easy application of the Cauchy-Schwarz inequality.
\begin{lemma}
\label{squares}
For $a_i \in \C$, $n \in \N$,
\[ |\sum_{i=1}^n a_i|^2 \leq n\sum_{i=1}^n |a_i|^2\]
\end{lemma}

We are now ready to prove the theorem.

\begin{proofof}{Theorem~\ref{D1D2_RD}}
Suppose that D1 and D2 hold, 
and let $K,P_1(x),P_2(x),P_3(x)$ and the sets
$\PP$, $S(g,k,l)$ and $T(k,l)$ be as specified by those conditions.

Choose $k,l,m$ such that $|k-l| \leq m \leq k+l$, and
define $\hat{k}:= \min(k,l)$.
Let $g\in G$ with $\ell(g)=m$.
Then
\[ |(\phi_k * \psi_l)(g)|
= |\sum_{\scriptsize\begin{array}{c} 
g_1\in C_k,g_2\in C_l\\ g=g_1g_2\end{array}}\phi_k(g_1)\psi_l(g_2)|
\leq  \sum_{\scriptsize\begin{array}{c} 
g_1\in C_k,g_2\in C_l\\ g=g_1g_2\end{array}}|\phi_k(g_1)\psi_l(g_2)|\]
Now the condition D2 defines an injection from
$\Factors_{k,l}(g)$ to $G^5 \times \N^2$
that maps $(g_1,g_2)$ to $(f_1,\hat{g},f_2,h_1,h_2,p_1,p_2)$,
where $h_1 \in C_{p_1}$, $h_2 \in C_{p_2}$, $(f_1,h_1),(h_2,f_2) \in \PP$
and $h_1h_2=\hat{g}$. We define 
\[ H(f_1,\hat{g},f_2,p_1,p_2) := \{h_1\in C_{p_1}: h_2:= h_1^{-1}\hat{g} \in C_{p_2}, (f_1,h_1),(h_2,f_2) \in \PP\}. \]
So the right hand side of the above inequality is bounded above by:
\[ 
\sum_{p_1,p_2=0}^{K\hat{k}} \, \sum_{ (f_1,\hat{g},f_2) \in S(g,k,l)} 
\sum_{
\scriptsize\begin{array}{c}h_1 \in H(f_1,\hat{g},f_2,p_1,p_2)\\
h_2:= h_1^{-1}\hat{g}\end{array}}
|\phi_k(f_1h_1)\psi_l(h_2f_2)|\]
Note that this summation is over a set that is
possibly larger than the image of that injection,
and hence we have an upper bound rather than equality.

Now by the Cauchy-Schwartz inequality, this sum is at most
\begin{eqnarray*}
&&
\sum_{p_1,p_2=0}^{K\hat{k}} \, \sum_{ (f_1,\hat{g},f_2) \in S(g,k,l)} 
\sqrt{ \sum_{
\scriptsize\begin{array}{c} h_1 \in C_{p_1}\\(f_1,h_1)\in \PP\end{array}}
|\phi_k(f_1h_1)|^2}
\sqrt{\sum_{
\scriptsize\begin{array}{c} h_2 \in C_{p_2}\\(h_2,f_2)\in \PP\end{array}}
|\psi_l(h_2f_2)|^2}\\
&=&
\sum_{p_1,p_2=0}^{K\hat{k}} \, \sum_{ (f_1,\hat{g},f_2) \in S(g,k,l)} 
\phi^{(p_1)}_{\PP,k-p_1}(f_1)
\quad{}^{(p_2)}\psi_{\PP,l-p_2}(f_2)\\
\end{eqnarray*}
where $\phi^{(p)}_{\PP,k-p}$ and ${}^{(p)}\psi_{\PP,l-p}$ are as defined above.

So now, we have
\begin{eqnarray*}
||(\phi_k * \psi_l)_m||_2^2 &=& \sum_{g \in C_m}|\phi_k * \psi_l(g)|^2\\ 
&\le&  \sum_{g \in C_m}|\sum_{p_1,p_2=0}^{K\hat{k}} 
 \sum_{
(f_1,\hat{g},f_2) \in S(g,k,l)} 
  \phi^{(p_1)}_{\PP,k-p_1}(f_1) \quad{}^{(p_2)}\psi_{\PP,l-p_2}(f_2)|^2
\end{eqnarray*}
Using Lemma~\ref{squares} twice we see that
\begin{eqnarray*}
&&  \sum_{g \in C_m}|\sum_{p_1,p_2=0}^{K\hat{k}} 
 \sum_{
(f_1,\hat{g},f_2) \in S(g,k,l)} 
  \phi^{(p_1)}_{\PP,k-p_1}(f_1) \quad{}^{(p_2)}\psi_{\PP,l-p_2}(f_2)|^2\\
&\le&  (K\hat{k}+1)^2 \sum_{g \in C_m}\sum_{p_1,p_2=0}^{K\hat{k}}
| \sum_{
(f_1,\hat{g},f_2) \in S(g,k,l)} 
  \phi^{(p_1)}_{\PP,k-p_1}(f_1) \quad{}^{(p_2)}\psi_{\PP,l-p_2}(f_2)|^2\\
&\leq& (K\hat{k}+1)^2 \sum_{p_1,p_2=0}^{K\hat{k}}
   \sum_{g \in C_m}  |S(g,k,l)|  \sum_{
(f_1,\hat{g},f_2) \in S(g,k,l)} 
 |\phi^{(p_1)}_{\PP,k-p_1}(f_1) \quad{}^{(p_2)}\psi_{\PP,l-p_2}(f_2)|^2\\
&\leq& (K\hat{k}+1)^2 P_2(\hat{k}) \sum_{p_1,p_2=0}^{K\hat{k}}
\sum_{g \in C_m}  \sum_{
(f_1,\hat{g},f_2) \in S(g,k,l)} 
 |\phi^{(p_1)}_{\PP,k-p_1}(f_1) \quad{}^{(p_2)}\psi_{\PP,l-p_2}(f_2)|^2
\end{eqnarray*}
The last sum is bounded above by
\[(K\hat{k}+1)^2 P_2(\hat{k})  \sum_{p_1,p_2=0}^{K\hat{k}}
  \sum_{\hat{g} \in T(k,l)} \left(\sum_{f_1 \in C_{k-p_1}} |\phi^{(p_1)}_{\PP,k-p_1}(f_1)|^2 \quad
   \sum_{f_2 \in C_{l-p_2}} |^{(p_2)}\psi_{\PP,l-p_2}(f_2)|^2 \right)\]
which is a sum of the same terms but over a possibly larger set.
We bound the last sum above by
\begin{eqnarray*}
&&  (K\hat{k}+1)^2 P_2(\hat{k}) P_3(\hat{k})
\sum_{p_1,p_2=0}^{K\hat{k}}
|\phi^{(p_1)}_{\PP,k-p_1}||_2^2 ||{}^{(p_2)}\psi_{\PP,l-p_2}||_2^2\\
&\leq&  
 (K\hat{k}+1)^2P_2(\hat{k})P_3(\hat{k})
  \sum_{p_1,p_2=0}^{K\hat{k}}
\FF_{\PP,k-p_1,p_1}\FF_{\PP,p_2,l-p_2} ||\phi_k||_2^2 ||\psi_l||_2^2\\
&\leq&
 (K\hat{k}+1)^2P_2(\hat{k})P_3(\hat{k})
  \sum_{p_1,p_2=0}^{K\hat{k}}
P_1(\min(k-p_1,p_1))P_1(\min(p_2,l-p_2)) ||\phi_k||_2^2 ||\psi_l||_2^2,
\end{eqnarray*}
using Lemma~\ref{relate_l2_norms} to relate the $l^2$-norms.
Since $\min(k-p_1,p_1)$ and $\min(p_2,l-p_2)$ above are bounded by $K\hat{k}$,
condition $(**)$ now follows easily.
\end{proofof}

\section{Notation for Artin groups}
\label{notation}
Our Artin groups will be defined in their standard presentations
\[ \langle x_1,\ldots ,x_n \mid {}_{m_{ij}}(x_i,x_j)=
{}_{m_{ij}}(x_j,x_i)\quad\hbox{\rm for each}\quad i \neq j \rangle. \]

We prove rapid decay for Artin groups satisfying the hypotheses of
Theorem~\ref{main_rd} with respect to the word length metric,
by verifying the conditions D1 and D2 above.
In order to do that we need to examine the structure of geodesics in
these groups; we build on the results of \cite{HoltRees}.

Given an Artin group $G$ as before we let $X$ be the set of generators
in the standard presentation and let $A$ be the set $X \cup X^{-1}$;
we call the elements of $A$
{\em letters}.
We shall generally use symbols like $x,y,z,t$ for generators in $X$ and
$a,b$ for letters in $A$.
A letter is {\em positive} if it is a generator, {\em negative} otherwise.
For each $x \in X$, we define $x$ to be the
{\em name} of the two letters $x$ and $x^{-1}$.
We say that a word $w \in A^*$ {\em involves} the generator $x$ if
$w$ contains a letter with name $x$, and we call $w$ a
{\em 2-generator} word if it involves exactly two of the generators.
Words in $A^*$ will be denoted by $u,v,w$ (possibly with subscripts) or
$\alpha,\beta,\gamma,\eta, \xi$. (Roughly speaking, the difference is that $u,v,w$
will be used for interesting subwords of a specified word, and the Greek
letters for subwords in which we are not interested.)
A {\em positive word} is one in $X^*$ and a {\em negative word} one in $(X^{-1})^*$; otherwise it is {\em unsigned}.
For $u,v \in A^*$, $u=v$ denotes equality as words, whereas $u=_Gv$ denotes
equality within the Artin group. The length of the word $w$ is denoted by $|w|$,
while as above $|w|_G$ denotes the length of a geodesic representative
and, for $g \in G$, $|g|$ denotes the length of a geodesic word representing
$g$.  A {\em syllable} in a word is a subword that is a power of a single
generator, and is not a subword of a higher power.

For any distinct letters $a$ and $b$ and a positive integer $r$, we define
alternating products ${}_r(a,b)$ and $(b,a)_r$. The product
${}_r(a,b)$, is defined, as it was earlier (over generators), to be the
word of length $r$ of alternating $a$'s and $b$'s starting
with $a$, while $(b,a)_r$ is defined to be the word of length $r$ of
alternating $a$'s and $b$'s ending
with $a$.
For example, $_6(a,b) = ababab = (a,b)_6$, $_5(a,b) = ababa = 
(b,a)_5$.
We define both $_0(a,b)$ and $(b,a)_0$ to be the empty word.
For any nonempty word $w$, we define $\f{w} $ and $\l{w} $ to be
respectively the first and last letter of $w$, and $\p{w}$  and $\s{w}$ to be
the maximal proper prefix and suffix of $w$. So $w=\p{w} \l{w} =\f{w} \s{w} $.

\section{Dihedral Artin groups}
\label{dihedral}
In this section, we prove that dihedral Artin groups of large type satisfy RD
by verifying the properties D1 and D2 in these groups.
As mentioned in the introduction, the fact that these groups satisfy
RD can be deduced immediately from results already in the literature, but the
techniques that we develop here are needed in the proof of Theorem
\ref{main_rd} in Section~\ref{largetype}.

\subsection{The structure of geodesics in dihedral Artin groups}
\label{sec:dihedral_geodesics}
We need some results on the structure 
of geodesics
from \cite{MairesseMatheus,HoltRees}.
We summarise what we need, and refer to those articles for details.

Let \[ \DA{m} = \langle x_1,x_2 \mid {}_m(x_1,x_2) = {}_m(x_2,x_1) \rangle \] be a
2-generator (dihedral) Artin group with $2 \le m < \infty$.
Conjugation by the Garside element \[ \Delta:={}_m(x_1,x_2)=_{\DA{m}}{}_m(x_2,x_1)\]
induces a permutation $\delta$ of order 2 or 1 (depending on the parity of $m$) on the
letters in $A$, and hence an automorphism $\delta$ of order 2 or 1 of the
free monoid $A^*$.

Let $w$ be a freely reduced word over $A=\{x_1,x_2,x_1^{-1},x_2^{-1}\}$.
Then we define $p(w)$ to be the minimum of $m$ and the length of
the longest subword of $w$ of alternating $x_1$'s and $x_2$'s (that is the
length of the longest subword of $w$ of the form $_r(x_1,x_2)$ or $_r(x_2,x_1)$).
Similarly, we define $n(w)$ to be the minimum of $m$ and the length of
the longest subword of $w$ of alternating $x_1^{-1}$'s and $x_2^{-1}$'s.
It is proved in \cite[Proposition 4.3]{MairesseMatheus} that $w$ is geodesic in $\DA{m}$ if and only if
$p(w) +n(w) \leq m$;
if $p(w)+n(w) < m$, then $w$ is the unique geodesic representative of the group
element it defines, but if $p(w)+n(w)=m$ then there are other representatives.
Note that a non-geodesic word is always unsigned. 

For example, suppose that $m=3$, so
$\DA{m}= \langle x_1,x_2 \mid x_1x_2x_1=x_2x_1x_2 \rangle.$
Then $w:=x_1x_2x_1$ and $w':=x_2x_1x_2$ are two geodesic representatives of
the same element with $p(w)=p(w')=3,n(w)=n(w')=0$, and
$w=x_1x_2^2x_1^{-1}$ and 
$w'=x_2^{-1}x_1^2x_2$ are two geodesic representatives of the same element
with $p(w)=p(w')=2,n(w)=n(w')=1$. In fact all four words are examples of {\em
critical words}, as defined in \cite{HoltRees}.

As in \cite{HoltRees},
we define a freely reduced word $w$ with $p(w)+n(w)=m$ to be critical if
one of the following holds, where $\xi$ is a subword with the obvious restrictions.
\begin{description}
\item[(i)] $w$ is unsigned of one of the two forms
${}_p(x,y)\xi (z^{-1},t^{-1})_n$ or
${}_n(x^{-1},y^{-1})\xi (z,t)_p$, with
$\{x,y\} = \{z,t\}=\{x_1,x_2\}$,
\item[(ii)] $w$ is positive of one of the two forms
${}_m(x,y) \xi$ or $\xi (x,y)_m$, and only the one positive alternating
subword of length $m$,
\item[(iii)] $w$ is negative of one of the two forms
${}_m(x^{-1},y^{-1}) \xi$ or $\xi (x^{-1},y^{-1})_m$, and only
the one negative alternating
subword of length $m$.
\end{description}

An involution $\tau$ on the set of critical words swaps critical words that
represent the same element.
We define $\tau$ by
\begin{eqnarray*}
{}_p(x,y)\,\xi\,(z^{-1},t^{-1})_n &\leftrightarrow^\tau&
{}_n(y^{-1},x^{-1})\,\delta(\xi )\,(t,z)_p,\\
{}_m(x,y)&\leftrightarrow^\tau& {}_m(y,x),\\
{}_m(x^{-1},y^{-1})&\leftrightarrow^\tau& {}_m(y^{-1},x^{-1})\\
{}_m(x,y)\,\xi &\leftrightarrow^\tau& \delta(\xi)\,(z,t)_m,
\quad\hbox{\rm where}\quad z=\l{\xi},\,\{x,y\}=\{z,t\},\\ 
{}_m(x^{-1},y^{-1})\,\xi &\leftrightarrow^\tau& \delta(\xi)\,(z^{-1},t^{-1})_m,
\quad\hbox{\rm where}\quad z=\l{\xi}^{-1},\,\{x,y\}=\{z,t\}.
\end{eqnarray*}

Where a critical word $w'$ occurs as a subword of a word $w$, we call the
substitution of the subword $w'$ by $\tau(w')$ a $\tau$-move on $w$.

The following lemma is proved in \cite[Lemma 2.3]{HoltRees}.

\begin{lemma}\label{geodlem}
Suppose that $w \in A^*$ is geodesic and $a \in A$.
If $wa$ is non-geodesic, then either $\l{w}=a^{-1}$
or $w$ has a critical suffix $v$ such that $\l{\tau(v)}=a^{-1}$.
Similarly, if $aw$ is non-geodesic, then either $\f{w}=a^{-1}$ or $w$ has
a critical prefix $v$ such that  $\f{\tau(v)}=a^{-1}$.
\end{lemma}

It is proved in \cite{HoltRees} that,
whenever $w$ is a freely reduced word that is not minimal
under the shortlex ordering, $w$ has a factorisation as $w_1w_2w_3$,
where $w_2$ is critical and
either $w_1\tau(w_2)w_3 \lexle w$ or $w_1\tau(w_2)w_3$ is not freely
reduced. In that case, we call the $\tau$-move on $w$ that replaces $w_2$ by
$\tau(w_2)$ together with any subsequent free reduction within $w_1\tau(w_2)w_3$ a
{\rm critical reduction} of $w$.
It follows from \cite[Theorem 2.4]{HoltRees} that
a succession of critical reductions reduces any word to its shortlex
minimal representative. Further, any two geodesic representatives are related by a 
sequence of $\tau$-moves.
In particular this implies that two representatives of the same group element
must either both be signed or both unsigned.

In order to deal with the reduction of non-geodesic words, we extend the
concept of $\tau$-moves.

A freely reduced  subword $u$ of a word $w$ is said to
be {\em over-critical} if it is has either of the forms
\[ {}_p(x,y)\xi (z^{-1},t^{-1})_n \quad{\rm or}\quad
 {}_n(x^{-1},y^{-1})\xi (z,t)_p, \]
with $p,n \le m$, $p+n > m$ and $\{x,y\} = \{z,t\}=\{x_1,x_2\}$,
and the extra condition that, if $p<m$, then $_p(x,y)$ or $(z,t)_p$
is a maximal positive alternating subword of $w$ and if $n<m$ then
$(z^{-1},t^{-1})_n$ or $_n(x^{-1},y^{-1})$ is a maximal negative alternating
subword
of $w$.

Note that we do not require that $p=p(w)$ or $n=n(w)$. But since
the conditions on $p,n$ force $p(w)+n(w)>m$, over-critical words
are necessarily non-geodesic.

We define $\tau$ on over-critical words by
\begin{eqnarray*}
\tau({}_p(x,y)\,\xi\,(z^{-1},t^{-1})_n)
&:=& {}_{m-p}(y^{-1},x^{-1})\,\delta(\xi )\,(t,z)_{m-n},\\
\tau({}_n(x^{-1},y^{-1})\,\xi\,(z,t)_p) 
&:=& {}_{m-n}(y,x)\,\delta(\xi )\,(t^{-1},z^{-1})_{m-p}.
\end{eqnarray*}
We refer to these as length reducing $\tau$-moves.
Note that any such move can also be achieved by an ordinary $\tau$-move
followed by free reduction.
So it follows from Lemma~\ref{geodlem} that any word can be reduced to a
geodesic by a sequence of length reducing $\tau$-moves and free reduction.
We call the move {\em positive} if $p=m$, {\em negative} if $n=m$ and
{\em unsigned} otherwise. (So if $p=n=m$ it is both positive and negative, but
we won't need to use that case.)

\begin{lemma}\label{deltafree}
Let $w$ be freely reduced with $0<p(w),n(w)<m$. Then $w$ can be
reduced to a geodesic using a sequence of unsigned length reducing
$\tau$-moves alone (i.e. with no free reduction).
\end{lemma}
\begin{proof} This follows from lemma~\ref{geodlem} together with the fact that applying an unsigned length
reducing $\tau$-move to $w$ results in a freely reduced word, and does not
increase $p(w)$ or $n(w)$.
\end{proof}

We shall also need the following lemma when we come to verify D2 for Artin
groups of extra large type.

\begin{lemma}\label{nclemma1}
Let $g,h \in DA(m)$.
If $gh \in \langle x_i \rangle$, then
$g$ and $h$ have geodesic factorisations $x_i^sw$ and $w^{-1}x_i^t$,
for some element $w$ and integers $s,t$.
\end{lemma}
\begin{proof} 
The result is trivial if either $g$ or $h$ (and hence both) is in
$\langle x_i\rangle$, so we suppose not.

Suppose that $g = a^k h^{-1}$ with $a = x_i^{\pm 1}$ and $k \ge 0$.
Suppose first that $|ah^{-1}| = 1 + |h|$. Then $|g| = k + |h|$ and the result
is clear. Otherwise, by Lemma~\ref{geodlem}, $a^{-1}$ is a left divisor of
$h^{-1}$. Let $l$ be maximal such that $a^{-l}$ is a left divisor of $h^{-1}$,
and let $h^{-1} = a^{-l} w$.
Then $a^{k-l} w$ is a geodesic factorisation of $g$, and the
result follows.
\end{proof}

\subsection{Verifying D1 and D2 for dihedral Artin groups}
\label{sec:dihedral_D1D2}
Our aim in this section is to verify that the properties D1 and D2 hold in any
dihedral Artin group $G$.
In fact, it can be shown that the  kernel of the natural homomorphism of
$\DA{m}$ onto the dihedral group of order $2m$ in which the images of $x_1$ and
$x_2$ have order 2 is a direct product of an infinite cyclic group
and a free group of rank $m-1$. The fact that $\DA{m}$ has rapid decay then
follows from \cite[Propositions 2.1.5, 2.1.9]{J90}, so we do not need
to verify D1 and D2 in order to prove rapid decay in these groups.  However we need
the properties here in order to prove that corresponding properties hold for
Artin groups satisfying the hypotheses of Theorem~\ref{main_rd}.

We shall assume throughout this section that $m \ge 3$.
We assume also that $k \leq l$, and deduce D1 and D2
in that case; the case $k \geq l$ then follows immediately by symmetry.

We first need to define our set $\PP$ of permissible geodesic factorisations of
elements $g \in G$. For unsigned elements $g$, we define $\PP(g)$ to be the set
all geodesic factorisations of $g$; that is,
\[ \PP(g) := \{(g_1,g_2) \in G^2 : g_1g_2=g,\,|g_1| + |g_2| = |g|\}.\]
If $m$ is infinite (that is, the 2-generator group is free) then for any
element $g$ we define $\PP(g)$  to be its set of geodesic factorisations.
From now on we shall assume that $m < \infty$.

Unfortunately, for $m<\infty$,
if we adopt the above definition of $\PP(g)$ for signed elements $g$, then D1
does not hold, so we are forced to use a more restrictive definition, which
significantly increases the technical complications in the proofs.

For a positive (respectively negative) element $g$, we define $d(g)$ to be
the maximal $k$ such that $\Delta^k$ (respectively $\Delta^{-k}$) is a
divisor of $g$. Then we call a geodesic factorisation $g_1g_2 = g$ of $g$ 
{\em $\Delta$-decreasing} if $d(g_1) + d(g_2) < d(g)$.
We define $\PP(g) := \PP^1(g) \cup \PP^2(g)$, where $\PP^1(g)$ is the
set of factorisations $g_1g_2$ for which at least one of $g_1,g_2$ is
represented by geodesic words with at most two syllables
(i.e. either $g_1$ or $g_2$ equals $a^sb^t$ with $a,b \in A$,
$s,t \ge 0$), and $\PP^2(g)$ is the set of geodesic factorisations of
$g$ that are not $\Delta$-decreasing.  That is,
$$\PP^2(g) := \{(g_1,g_2) \in G^2 : g_1g_2=g,\,|g_1| + |g_2| = |g|, \,
d(g_1)+d(g_2)=d(g) \}.$$
(In fact we could omit the factorisations in $\PP^1$ and still obtain
D1 and D2 in the dihedral case, but we will need them in the next section in
order to prove the conditions when there are more than two generators;
that proof is easier if we include those
factorisations already for the dihedral case.)
We say that a left (respectively right) divisor $h$ of $g$ lies in
$\PP^i_l(g)$ (respectively $\PP^i_r(g)$) for $i=1$ or 2 if
$(h,h^{-1}g) \in \PP^i(g)$ (respectively $(gh^{-1},h) \in \PP^i(g)$),
and let $\PP_l(g) := \PP^1_l(g) \cup \PP^2_l(g)$ and
$\PP_r(g) := \PP^1_r(g) \cup \PP^2_r(g)$.

Our first aim is to prove Property D1.
In order to do that we need to examine the set of geodesic words 
representing a given element.
Let $w$ be a geodesic word with $p(w)+n(w)=m$;
we write $p:= p(w), n:= n(w)$. We have seen that
$w$ can be reduced to its shortlex normal form by a sequence of lex-reducing
critical reductions. It follows that if $w'$ is another geodesic word with
$w=_G w'$, then $w$ can be transformed to $w'$ by a sequence of $\tau$-moves.

First suppose that $p,n > 0$.
The word $w$ has the form $\eta_0w_1\eta_1 \cdots w_s \eta_s$,
where each $w_i$ has the form $(x,y)_p$ or $(x^{-1},y^{-1})_m$ with
$\{x,y\} = \{x_1,x_2\}$, and $p(\eta_i)<p$, $m(\eta_i)<m$ for all $i$.
A $\tau$-move has the effect of changing the signs of $w_i$, $w_j$ for some
$i<j$ for which $w_i,w_j$ have opposite signs, and replacing the subword
between them by its image under $\delta$.  Denote this $\tau$-move by $(i,j)$.
So $w' = \eta'_0w'_1\eta'_1 \cdots w'_s \eta'_s$, where each
$\eta'_r = \eta_r$ or $\delta(\eta_r)$, and all of the subwords $w_r,w'_r$
are maximal alternating.
But note that although $|\eta'_r|=|\eta_r|$, it is not necessarily
true that $|w_r|=|w'_r|$.

\begin{lemma}
\label{unsigned_lem} Let $w,w'$ be geodesic words as above.
If, for some r with $1 \le r \le s$, the number of positive
$w_i$ with $1 \le i \le r$ is equal to the number of positive $w'_i$
 with $1 \le r \le s$, then $\eta'_0w'_1\eta'_1 \cdots w'_r =_G
\eta_0w_1\eta_1 \cdots w_r$ and $\eta_r=\eta'_r$.
\end{lemma}
\begin{proof}  This is by induction on $r$.
Since none of the $\tau$-moves changes $\eta_0$, we have $\eta_0=\eta'_0$, and
so the result is true for $r=0$.

So suppose that $r \ge 1$.
Since $w_1$ can only be changed by a transformation $(1,i)$ with $1<i$,
there are only two possible $w_1'$, one of which is $w_1$ and the
other a word of opposite sign to $w_1$.
So if $\sgn(w_1) = \sgn(w'_1)$, then $w_1=w_1'$ and the result follows by
induction. Otherwise, we have $\{w_1,w_1'\} =
\{ {}_p(x,y),\,{}_m(y^{-1},x^{-1}) \}$.
From the hypothesis, there must be both positive and negative $w_i$
with $i\le r$, so there is a $w_i'$ with $2 \le i \le r$ and
$\sgn(w_i') = -\sgn(w_1')$. We can do the move $(1,i)$ to $w'$ giving $w''$.
Such a move does not change the group element
$\eta'_0w'_1\eta'_1 \cdots w'_r$ or $\eta'_r$, but it does change
$w_1'$ back to $w_1$,
so the result follows by induction applied to $w, w''$.
\end{proof}

For a positive element $g$ (the negative case is similar), let $r=d(g)$.
Then, for a geodesic factorisation $g = g_1g_2$ in $\PP^2(g)$, we have
$d(g_1) = s$ and $d(g_2) = r-s$ for some $0 \le s \le r$, and then
$g_1$ has $\Delta^s$ as a left divisor and $g_2$ has $\Delta^{r-s}$ as
a right divisor.

\begin{lemma}
\label{signed_lem} Let $w,w'$ be positive geodesic words representing
$g \in G$ such that $w$ and $w'$ both have $\Delta^s$ as a prefix and
$\Delta^{r-s}$ as suffix for some $s$, where $r = d(g)$. Then $w=w'$.
\end{lemma}
\begin{proof} We have $w = \Delta^s \eta  \Delta^{r-s}$,
$w' = \Delta^s \eta'  \Delta^{r-s}$ with $\eta =_G \eta'$. Since $d(g)=r$,
we have $d(\eta) = d(\eta')=0$ so by~\cite[Proposition 4.3]{MairesseMatheus}
$\eta_G$ has
a unique geodesic representative, and hence $\eta=\eta'$ and $w=w'$.
\end{proof}

\begin{lemma}\label{p1count}
Let $g \in G$ be a positive element and let $l \ge 0$. Then, for $a,b \in A$
with $\{a,b\} = \{x_1,x_2\}$, the number of right divisors of $g$ of the form
$a^sb^t$ with $s+t=l$ is at most $d(g)+1$.
\end{lemma}
\begin{proof}
The proof is by induction on $l$, the case $l=0$ being trivial.
We may suppose that $a^l$ is a right divisor of $g$, since otherwise all right
divisors of $g$ of the required form end in $b$, and the result follows from
the case $l-1$ applied to $gb^{-1}$.
Similarly, we may assume that $b^l$ is a right divisor of $g$, since
otherwise the suffixes of the required form all begin with $a$, and
the result follows from the case $l-1$ applied to $g$.

Let $d=d(g)$. We may assume that $d>0$, since otherwise $g$ has a unique
geodesic representative.
It is straightforward to prove by induction on $t$ that, for any $t>0$,
the group element $g_{a,t} := \Delta^t a^{-t}$ is positive with $d(g_{a,t}) = 0$ and
$\l{g_{a,t}} = b$ (that is, the unique geodesic representative ends in $b$);
in fact the unique geodesic representative is a concatenation of positive
alternating words each of length $m-1$. 
Similarly $g_{b,t} = \Delta^t b^{-t}$ has $d(g_{b,t})=0$
with $\l{g_{b,t}}=a$ and, since $g_{b,t}$ is obtained from $g_{a,t}$ by
interchanging $a$ and $b$, we have $\f{g_{a,t}} \ne \f{g_{b,t}}$.
We write $g = g' \Delta^d$ with $d(g')=0$, and observe that
the element $g'g_{a,d}=ga^{-d}$ is divisible by $\Delta$ if and only if $g'$ ends in a 
positive letter other than $\f{g_{a,t}}$, and similarly for the element
$g'g_{b,d}=gb^{-d}$; hence exactly one
of the two group elements $g' g_{a,d}$ and $g' g_{b,d}$
is not divisible by $\Delta$.
If $g' g_{a,d} = ga^{-d}$ is not divisible by $\Delta$, then its unique
geodesic representative ends in $\l{g_{a,d}}=b$; $ga^{-d}$ cannot have $a$ as a right
divisor, and $g$ cannot have $a^{d+1}$ as a right divisor;
Similarly we see that if $g'g_{b,d}$ is not divisible by $\Delta$,
then $g$ cannot have $b^{d+1}$ as a right divisor.
Hence $g$ cannot have both $a^{d+1}$ and $b^{d+1}$
as right divisors.
It now follows from the preceding paragraph that $l \le d$, and then the result follows immediately, by counting the number of pairs
$(s,t)$ with $s+t=l$.
\end{proof}

We can now prove D1, which we can express as follows:

\begin{corollary}\label{geodcor}
There is a polynomial $P(x)$ with the following property:
for any $g \in G$ the number of left divisors of $g$ of length $k$ in
$\PP_l(g)$ is bounded above by $P(k)$.
\end{corollary}
\begin{proof} 
Suppose that $g \in G$ has length $k+l$.
We need to bound the number of group elements $u_G$ for which $w=uv$ is a
geodesic word representing $g$, with $|u|=k,|v|=l$ and $(u_G,v_G) \in \PP$.

First suppose that $w=uv$ is unsigned, of the form
$\eta_0w_1\eta_1\cdots w_s\eta_s$, as above. If one of
the subwords $w_i$ of $w$ intersects both $u$ and $v$, then
let $u_1$ be the prefix of $w$ that ends at the end of $w_i$; otherwise
let $u_1= u$.

Suppose that $u_1$ contains $r_p$ positive subwords $w_i$ and $r_n$
negative subwords $w_i$. The group element $(u_1)_G$ is determined
by those two parameters by Lemma~\ref{unsigned_lem}.
Since each of $r_p,r_n$ must be in the range $[0,k]$, there are at most
$k^2$ choices for $(u_1)_G$ corresponding to unsigned geodesics $w$.
Since $u$ differs from $u_1$ by an alternating
subword of length less than $m$ (and there are less than $4m$ such words),
there are at most $4mk^2$ choices for $u_G$ that
correspond to unsigned geodesic representatives $w$.

Now suppose that $w=uv$ is signed. Suppose first that $(u,v) \in \PP^1$,
so that either $u$ or $v$ has at most 2 syllables.
There are at most $2(k+1)$ possible words $u$ with at most two
syllables. By Lemma~\ref{p1count}, there are at most $2(d(g)+1)$
possible words $v$ with at most two syllables.
If the signed element $g$ has the geodesic factorisation $g=g'a^sb^t$
then, by collecting powers of $\Delta$ that divide $g$ to the left, we
see that at least $m-2$ of the letters in each occurrence of $\Delta$
in $g$ must come from $g'$, and hence $|g'| \ge (m-2)d(g) \ge d(g)$.
So, if there are factorisations in which $v$ has two syllables, then
$k \ge d(g)$.  So there are at most $4(k+1)$ such pairs $(u,v)$.

Now assume that $(u,v) \in \PP^2$.  Then $u=_G \Delta^su'$ and
$v=_G v'\Delta^{d(g)-s}$ for some $s$, where $u'v'$ is not divisible by
$\delta$, and by Lemma~\ref{signed_lem} $u'v'$ is completely determined
by $g,s$. Further $u'v'$ is the unique geodesic representative of $(u'v')_G$.
Hence (given $g,k$), $(u_G)$ is completely determined by $s \in [0,k]$.  
The number of such factorisations is therefore at most $k+1$.
\end{proof}

In order to verify the condition D2, we first describe a process that
we call {\em merging}, which we can use to define the set $S(g,k,l)$ that
appears in that condition.

Given elements $g_1,g_2$ of length $k,l$ whose product $g$ has length
less than $k+l$, an application of the  merging process
results in a triple $(f_1,\Delta^r,f_2)$ of elements, such that
for some $h_1,h_2$ with $h_1h_2=_G \Delta^r$, $f_1h_1$ and $h_2f_2$ are
geodesic factorisations of $g_1,g_2$, respectively, and furthermore
$(f_1,h_1),(h_2,f_2) \in \PP$.
We call such a triple a {\em merger}
of $g_1$ and $g_2$; we do not claim or need $g_1,g_2$ to
have a unique merger (although we suspect that it does).
Then we define the set $S(g,k,l)$ to be the set of all triples
$(f_1,\Delta^r,f_2)$ that arise as mergers of pairs of elements $g_1,g_2$
of lengths $k,l$ and with $g_1g_2=_G g$.

We compute a merger of $(g_1,g_2)$ as the last term of a sequence
of triples $(g_1^{(t)},\Delta^{r_t},g_2^{(t)})$, defined as follows.
When one or both of $g_1,g_2$ is a signed word, we have to be careful to ensure
that the resulting geodesic factorisations of $g_1$ and $g_2$ lie in $\PP$.
We set $g_1^{(1)}:=g_1$, $r_1:=0$, $g_2^{(1)}:=g_2$.

Now, the $t$-th step of the merging process computes
$g_1^{(t+1)},r_{t+1},g_2^{(t+1)}$ as follows.
In each of the three situations below, we choose non-identity group elements
$h,h'$, and then define $g_1^{(t+1)} :=  g_1^{(t)}h^{-1}$ and
$g_2^{(t+1)} := h'^{-1}g_2^{(t)}$ such that $g_1^{(t+1)}h$ and
$h' g_2^{(t+1)}$ are geodesic factorisations of $g_1^{(t)}$ and $g_2^{(t)}$
respectively, with $g_1^{(t+1)} \in \PP_l(g_1)$ and $g_2^{(t+1)} \in \PP_r(g_2)$.

\begin{mylist}
\item[(i)] If we can choose $h,h'$ with $h \delta^{r_t}(h') = 1$,
then we do so with $h,h'$ as long as possible, and put $r_{t+1}:= r_t$
(we call this a cancellation move).
\item[(ii)] Otherwise, if $g_1^{(t)}$ and $g_2^{(t)}$ are both signed
words, and we can choose $h=h' = \Delta^\epsilon$ with $\epsilon = \pm 1$,
then do so, and put $r_{t+1} := r_t + 2\epsilon$.
\item[(iii)] Otherwise, if we can choose $h,h'$ with
$h \delta^{r_t}(h') = \Delta^{\epsilon}$ with $\epsilon = \pm 1$, then do so
and put $r_{t+1}:= r_t+\epsilon$ (we call this a $\Delta$-extraction move).
\end{mylist}

If for some $t$, no pair of elements $h,h'$ satisfy any of these three conditions, then
merging is complete, and we output $(f_1,\Delta^r,f_2):=(g_1^{(t)},\Delta^{r_t},g_2^{(t)})$ as a merger of $(g_1,g_2)$.
Note that we needed the 
conditions at each previous step
that $g_1^{(t+1)} \in \PP_l(g_1)$ and
$g_2^{(t+1)} \in \PP_r(g_2)$ to ensure that
$(f_1,\Delta^r,f_2)$ possesses all the required properties of
a merger of $(g_1,g_2)$.

\begin{lemma}\label{mergerlemma}
Suppose that $g_1,g_2$ are elements of length $k,l$ respectively,
where $k \leq l$, and that $(f_1,\Delta^r,f_2)$ is a merger of $(g_1,g_2)$;
Let $h_1:= f_1^{-1}g_1$ and $h_2 := g_2 f_2^{-1}$.
Then $r \leq k$ and $|h_1|,|h_2| \leq (m-1)k$.
\end{lemma}
\begin{proof} 
It is clear that the process above completes in at most $k$
stages, giving the required bound on $r$.
At each stage of the merging process, the ratio of the lengths of the
elements $h,h'$, stripped  from $g_1^{(t)}$ and $g_2^{(t)}$ to
give $g_1^{(t+1)}$ and $g_2^{(t+1)}$, is in the
interval $[1/(m-1),m-1]$.
Hence the same is true for the total lengths of elements stripped off,
that is, $1/(m-1) \le |h_1|/|h_2| \le m-1$.
Since $|h_1| \le |g_1| = k$ and $h_2 \le |g_2| = l$, we get
$|h_1|,|h_2| \le (m-1)k$, as claimed.
\end{proof}

The definition of the merging process ensures that, for a merger
$(f_1,\Delta^r,f_2)$ of $(g_1,g_2)$, we have $(f_1,h_1),
(h_2,f_2) \in \PP$.
So, in order to verify that D2 holds, we just need to find polynomial bounds
on $S(g,k,l)$ and the associated set $T(k,l)$ of powers of $\Delta$, and a
value for the constant $K$.

The above lemma already bounds $|T(k,l)|$ by $2k+1$ and $K$ by $m-1$,
so it remains to find a polynomial bound on $S(g,k,l)$.
In order to do this we need to explore a reduction process, which we call
{\em compression}, that starts
with a triple $(f_1,\Delta^r,f_2) \in S(g,k,l)$ and produces a geodesic representative of $g$.
Reversing the compression process will then
enable us to estimate the size of $S(g,k,l)$.

So suppose that $(f_1,\Delta^r,f_2) \in S(g,k,l)$.
The fact that the merging process has completed at
$(f_1,\Delta^r,f_2)$ ensures
that $f_1,f_2$ cannot both have powers of $\Delta$ as divisors
and that, if one of the two elements has such a divisor $\Delta^{r_0}$
and $r,r_0\neq 0$, then $r,r_0$ cannot have opposite signs.
Hence we see that $f_1,f_2$ have geodesic representatives of the form
$u\Delta^{r_0},v$ or $u,\Delta^{r_0}v$ for some $r_0$ (possibly zero),
where $u,v$ are not divisible by $\Delta^{\pm 1}$.
Then $g_1\Delta^r g_2$ is represented by
$u\Delta^{r_1}v=_G u \delta^{r_1}(v)\Delta^{r_1}$, where $r_1=r+r_0$ has
the same sign as $r$, and (since the merging process terminated),
$u\delta^{r_1}(v)$ is freely reduced and also contains no powers of $\Delta$.
Then, by Lemma~\ref{deltafree}, if non-geodesic, $u\delta^{r_1}(v)$ 
can be reduced to a
geodesic word using only unsigned length reducing $\tau$-moves.

We describe that reduction more precisely as follows.
We define $u_1$ to be the shortest
prefix of $u\delta^{r_1}(v)$ that contains $u$ and ends at the end of a maximal
alternating subword, and define $v_1$ to be the remainder of the word.
Then $v_1$ is geodesic, and we set $v_2:= v_1$. If $u_1$ is also geodesic, then
we set $u_2:= u_1$, but otherwise
we set $u_2$ to be a geodesic word 
derived from $u_1$ by a single length reducing $\tau$-move.
We set $u_2^{(1)}:= u_2,v_2^{(1)}:= v_2$.

We can now reduce $u_2^{(1)}v_2^{(1)}$
to geodesic form through a series of words $u_2^{(t)}v_2^{(t)}$ using a sequence
of length reducing $\tau$-moves, the $i$-th of which
involves one alternating subword within $u_2^{(t)}$ and one within $v_2^{(t)}$.
Specifically, we derive $u_2^{(t+1)},v_2^{(t+1)}$ from $u_2^{(t)},v_2^{(t)}$
by either replacing a suffix 
${}_p(x,y)\,\xi_1$ in $u_2^{(t)}$ and a prefix $\xi_2\,(z^{-1},t^{-1})_n$
in $v_2^{(t)}$ by a suffix
${}_{m-p}(y^{-1},x^{-1})\,\delta(\xi_1)$ in $u_2^{(t+1)}$
and a prefix $\delta(\xi_2)\,(t,z)_{m-n}$ in $v_2^{(t+1)}$ or
by replacing a suffix
${}_n(x^{-1},y^{-1})\,\xi_1$ in $u_2^{(t)}$ and a prefix $\xi_2\,(z,t)_p$
in $v_2^{(t)}$ by a suffix 
${}_{m-n}(y,x)\,\delta(\xi_1 )$ in $u_2^{(t+1)}$
and a prefix $\delta(\xi_2)\,(t^{-1},z^{-1})_{m-p}$ in $v_2^{(t+1)}$.
Eventually this process produces a geodesic word $u_3v_3$.

Next we need to consider
the reduction to geodesic form of $u_3v_3\Delta^{r_1}$.
We do this in two stages. First we reduce $v_3\Delta^{r_1}$ to
$v_4\Delta^{r'}$, using a sequence of length reducing $\tau$-moves, each of which involves
a single $\Delta^{\pm 1}$ 
and an alternating subword of the opposite sign within $v_3$ (or the word
derived from it); at the end of 
this stage, either $v_4\Delta^{r'}$ is signed or $r'=0$. 

Now if $r'=0$ we set $u_4:=u_3$, but otherwise, we apply a further
sequence of length reducing $\tau$-moves to $u_3\Delta^{r'}$, each move
involving a single $\Delta^{\pm 1}$ and an alternating subword of the 
opposite sign in the word derived from $u_3'$, and hence reduce
$u_3\Delta^{r'}$ to $u_4\Delta^s$.  The fact that
$u_4\delta^{r'-s}(v_4)\Delta^s$ is now a geodesic representative
of $u_3v_3\Delta^{r_1}$ follows from the two equations
$u_3v_4\Delta^{r'}=_G u_3\Delta^{r'}\delta^{r'}(v_4)$
and $u_4\delta^{r'-s}(v_4)\Delta^s =_G u_4\Delta^s\delta^{r'}(v_4)$.
Our construction ensures that $u_4\delta^{r'-s}(v_4)$ has no divisor of the
form $\Delta^{\pm 1}$.
This is the end of compression;
we use the name $\kappa(f_1,\Delta^r,f_2)$ for the group element $(u_4)_G$. 

We now want to estimate both the number of group elements that can
arise as $\kappa(f_1,\Delta^r,f_2)$ out of the compression of at least 
one element $(f_1,\Delta^r,f_2)$ of $S(g,k,l)$,
and also the number of triples $(f_1,\Delta^r,f_2) \in S(g,k,l)$ for which
$\kappa(f_1,\Delta^r,f_2)=g'$ for 
a particular group element $g'$.
The product of these two values will
give a bound on $|S(g,k,l)|$.

\begin{lemma}
\label{bound_kappaS}
The size of $\kappa(S(g,k,l))$ is bounded by a polynomial in $k$.
\end{lemma}

\begin{proof}
Suppose that compression of $(f_1,\Delta^r,f_2)$ leads to
$u_4\delta^{r'-s}(v_4)\Delta^s$.
We first need to relate $|u_4|$ to $k=|g_1|$.
We recall that $|u|\leq |f_1|\leq k$, $|u_1| \leq |u|+m-1$,
and $|u_2| \leq |u_1|$.
Then each move from $u_2^{(t)}$ to $u_2^{(t+1)}$ replaces a distinct alternating
subword by an alternating subword that is longer by a factor of at most $m-1$,
and doesn't alter the lengths of the subwords before or after that alternating
subword.  Hence $|u_3| \leq (m-1)|u_2|$. By the same argument
$|u_4| \leq (m-1)|u_3|$. So $|u_4| \leq (m-1)^2(k+m-1)$.

Then $u_4$ is a prefix of length $k' \leq (m-1)^2(k+m-1)$
of $w' := u_4\delta^{r'-s}(v_4)$, for which $w'\Delta^s$ is a geodesic
representing $g$ and
no power of $\Delta$ divides $w'$; it follows that $u_4 \in \PP_l(g)$.
By Corollary~\ref{geodcor} we know that, for some polynomial $P_1(x)$,
$P_1(k')$ bounds the number of elements of $\PP_l(g)$
represented by a word that is a prefix of length $k'$ of some such $w$.
Hence, since $k'$ is bounded by a polynomial
in $k$, the number of group elements represented by any such
$u_4$ with $k'$ in the appropriate range is bounded by a polynomial in $k$.
\end{proof}

\begin{lemma}
\label{bound_each_preimage}
The number of triples $(f_1,\Delta^r,f_2)$ in $S(g,k,l)$ for
which $\kappa(f_1,\Delta^r,f_2)$ is a particular element $g' \in G$
is bounded by a polynomial in $k$.
\end{lemma}

\begin{proof}
We recall that for elements of $S(g,k,l)$, $f_1,f_2$ are
represented either by words
$u\Delta^{r_0},v$ or by words $u,\Delta^{r_0}v$, where $u,v$
have no divisors of the form $\Delta^{\pm 1}$.
The triple $(f_1,\Delta^r,f_2)$
is completely determined by the above binary choice together with $u_G,r$, and
$r_0$ in the first case.
But $r$ (and $r_0$ in the first case) is bounded above by $k$. Hence it is now
sufficient to bound the number of choices of $u_G$. Hence the proof of this
lemma is completed by application of the lemma that follows.
\end{proof}

\begin{lemma}
Given $g' \in G$, a polynomial in $k$ bounds
the number of group elements $u_G$ associated as above with triples
$(f_1,\Delta^r,f_2)$ in $S(g,k,l)$ for which $\kappa(f_1,\Delta^r,f_2)=_G g'$.
\end{lemma}

\begin{proof}
We reverse the steps of compression.
Suppose that $w=u_4\delta^{r'-s}(v_4)\Delta^s$ is a word resulting
from compression of some element of $S(g,k,l)$,
and that $(u_4)_G = g'$.

We first examine the possible 
geodesics of the form $u_3v_3\Delta^{r'}$ from which 
$w$ might have been derived during compression
by length reducing moves.

If $v_4\Delta^s$ is unsigned, then we must have $u_3=u_4$,
$r'=s$, and $(u_3)_G = (u_4)_G$
Otherwise, if $\delta^{r'-s}(v_4)\Delta^s$ is positive,
$r'-s$ might be any non-negative integer bounded above by
the number of maximal positive alternating subwords in $u_4$, and
$u_3$ could be any word derived from $u_4$ by reversing $r'-s$
length reducing rules using $\Delta$,
while if $\delta^{r'-s}(v_4)\Delta^s$ is negative,
$s-r'$ might be any non-negative integer
bounded above by
the number of maximal negative alternating subwords in $u_4$,
$u_3$ could be a word derived from $u_4$ by reversing $s-r'$
length reducing rules using $\Delta^{-1}$. 
We can check that 
the value of $(u_3)_G$ is determined by $(u_4)_G$ and $r'-s$.

Hence there are at most $(k+1)$ possibilities
for $(u_3)_G$ for which $(u_4)_G =_G g'$.

We saw above that if $v_3$ is unsigned or
if it has the opposite sign to $r'$, then $|r'| \le k$, and so can take up to
$2k+1$ values. 
Otherwise, $v_3 \Delta^{r'}$ is a signed word,
and then $r'$ is uniquely determined by $(v_3 \Delta^{r'})_G$.

We next reverse the length reducing $\tau$-moves that transformed
$u_2v_2$ to $u_3v_3$.
Each reversed move involves replacing a subword
${}_p(x,y)\xi (z^{-1},t^{-1})_n$ or\linebreak
${}_n(x^{-1},y^{-1})\xi (z,t)_p$
with $0 < p,n < m$ and $p+n<m$ by
${}_{m-p}(y^{-1},x^{-1})\,\delta(\xi )\,(t,z)_{m-n}$
or ${}_{m-n}(y,x)\,\delta(\xi )\,(t^{-1},z^{-1})_{m-p}$,
where the left maximal alternating subword is in some $u_2^{(t+1)}$ and
the right one in $v_2^{(t+1)}$. Call these transformations of type
$(p,-n)$ or $(-n,p)$.
Then there are at most $m^2$ such possible types.
The element $(u_2)_G$ depends only on $(u_3)_G$
and the total number of $\tau$-moves of each type.
Since there are at most $k$ such $\tau$-moves in total, there are at most
$k^{m^2}$ possible elements $(u_2)_G$ for a given $(u_3)_G$.

And $u_1$ represents the same element as $u_2$.
A final reversal of a $\tau$-move on the final maximal alternating
suffix of $u_2$ recovers the word $u_1$, which represents the same
group element as $u_2$. And the deletion of an alternating
suffix of length at most $m-1$ from $u_1$ yields $u$; hence 
each choice of $(u_1)_G$ yields at most $m$ choices of $u$.
\end{proof}
The combination of Lemma \ref{bound_kappaS} and Lemma \ref{bound_each_preimage}
 gives the required polynomial bound on $S(g,k,l)$.
Hence the proof of D2 for dihedral Artin groups is complete.

\section{Artin groups of large type}
\label{largetype}
Our aim in this section is to prove that Artin groups satisfying the
hypotheses of Theorem~\ref{main_rd} satisfy D1 and D2, and hence have rapid
decay.  Much of what we say is true for any large-type Artin group. But
occasionally we shall need to restrict the groups we consider to
those satisfying the conditions of Theorem~\ref{main_rd}. For those results
where this is the case, we shall make it clear in the statement of the result;
the remaining results are proved for all Artin groups of large type.

We assume throughout this section that $G$ is an Artin group of large type,
with notation as defined in Section~\ref{notation}.
We may assume that not all $m_{ij}$ are infinite, for otherwise the group is
free, and rapid decay is known.
For any distinct pair of generators $x_i,x_j$, we let
$G(i,j)=G(j,i)$ be the subgroup of $G$ generated by $x_i$ and $x_j$.
We use $\Delta_{ij}$ to denote $(x_i,x_j)_{m_{ij}}$,
and $\delta_{ij}$ to denote the permutation of $\{x_i,x_i^{-1},x_j,x_j^{-1}\}$
induced by conjugation by $\Delta_{ij}$.

The process of reducing words in $G$ to shortlex minimal form
is described in \cite[Proposition 3.3]{HoltRees}.
In \cite[Section 3]{HoltRees}, a leftward or rightward {\em critical
sequence} is defined as a sequence of $\tau$-moves applied to a word,
in which successive moves overlap in a single letter.
For a geodesic word whose maximal proper prefix is already shortlex minimal,
but which is not itself lexicographically minimal, a lexicographic reduction
to a shortlex minimal word can be achieved by  a single leftward critical
sequence, known as a {\em leftward lex reducing sequence}.
Similarly, for a non-geodesic word whose maximal proper prefix is shortlex
minimal, a length reduction to a shortlex minimal word can be achieved by a
single rightward critical sequence followed by a single free cancellation;
the combination is known as a {\em rightward length reducing sequence}.

For example, with $m_{12},m_{13},m_{23} = 3,4,5$ and writing $a,b,c$ for
$x_1,x_2,x_3$:
\begin{eqnarray*}
\alpha \mathbf{(a b^{-2} a^{-1})} c b^2 c^{-1} b^{-1} a c a^2 c a^{-1},\\
\alpha b^{-1} a^{-2} \mathbf{(b c b^2 c^{-1} b^{-1})} a c a^2 c a^{-1},\\
\alpha b^{-1} a^{-2} c^{-1} b^{-1} c^2 b \mathbf{(c  a c a^2 c)} a^{-1},\\
\alpha b^{-1} a^{-2} c^{-1} b^{-1} c^2 b a c^2 a c \mathbf{(a a^{-1})},\\
\alpha b^{-1} a^{-2} c^{-1} b^{-1} c^2 b a c^2 a c,
\end{eqnarray*}
where $\alpha$ is an arbitrary word,
is a rightward length reducing sequence in which the critical words to which
$\tau$-moves are applied are bracketed and printed in boldface,
and the final move is the free reduction of $a a^{-1}$.

\subsection{The geodesics in Artin groups of large type}
We shall need some results about geodesics in Artin groups of large type.
The first of these is proved in \cite{HoltRees}.

\begin{proposition}
[{\cite[Proposition 4.5]{HoltRees}}]
\label{twoposs}
Suppose that $v,w$ are any two geodesic words representing the same group
element, and that $\l{v} \ne \l{w}$.
Then:
\begin{mylist}
\item[(1)] $\l{v}$ and $\l{w}$ have different names;
\item[(2)] The maximal 2-generator suffixes of $v$ and $w$ involve generators
with names equal to those of $\l{v}$ and $\l{w}$;
\item[(3)]
Any geodesic word equal in $G$ to $v$ must end in $\l{v}$ or in $\l{w}$.
\end{mylist}
Corresponding results apply if $\f{v} \ne \f{w}$.
\end{proposition}

\begin{corollary}\label{twoposscor}
If $wa$ is a geodesic word for some $a \in A$ and $k>1$, then $wa^k$ is
a geodesic word for all $k>1$.
\end{corollary}
\begin{proof}
Otherwise, choosing $k$ minimal with $wa^k$ non-geodesic, the group
element $wa^{k-1}$ has geodesic representatives ending both in $a$ and
in $a^{-1}$, contradicting Proposition~\ref{twoposs} (1).
\end{proof}

\begin{proposition}\label{longestld}
Let $g \in G$ and $x_i,x_j \in X$ with $i \ne j$. Then $g$ has a unique left
divisor $\LD_{ij}(g) \in G(i,j)$ of maximal length. Furthermore, if $w$ is
any geodesic word representing $g$, and $u$ is the maximal
$\{x_i,x_j\}$-prefix of $w$,
then $\LD_{ij}(g) =_G u a^r$ for some $r \ne 0$ with
$a \in \{x_i^{\pm 1}, x_j^{\pm 1} \}$ and $|\LD_{ij}(g)| = |u| + r$.

Similarly, $g$ has a unique right divisor $\RD_{ij}(g) \in G(i,j)$ of
maximal length, to which the corresponding results apply.
\end{proposition}
\begin{proof} Order the monoid generators of $G$ such that $x_i^{\pm 1},
x_j^{\pm 1}$ (in some order) come first, and let $w'$ be the shortlex least
representative of $g$, using this ordering, with maximal $\{x_i,x_j\}$-prefix
$u'$.

Let $w$ be an arbitrary geodesic representative of $g$ and let $u$ be the
maximal $\{x_i,x_j\}$-prefix of $w$. We consider the process of reducing $w$ to
its shortlex form $w'$ by considering each letter of $w$ in turn, and
reducing the prefix ending in that letter to shortlex form. Suppose
that $w_0=w$, that $w$ reduces through the sequence of words $w_1,w_2,\ldots$
to $w_n=w'$ and that $u_0=u,u_1,\ldots, u_n=u'$ is the corresponding
sequence of maximal $\{x_i,x_j\}$-prefixes.

By \cite[Proposition 3.3]{HoltRees}, the prefix of length $k$ in $w_k$ 
is either already shortlex reduced, or can be reduced to shortlex
form with a single leftward lex reducing sequence.
Such a reduction cannot change a letter of $u_k$ to a letter with name
not in $\{x_i^{\pm 1}, x_j^{\pm 1} \}$, since that would be shortlex
increasing.  So either $u_{k+1}=u_k$ or $u_{k+1}$ is the
shortlex reduction of $u_ka^s$ for some $a \in \{x_i^{\pm 1}, x_j^{\pm 1 }\}$
and $s > 0$.
By \cite[Lemma 3.8(2)]{HoltRees}, after such a reduction, the names of
the next two letters in $w_{k+1}$ after $u_{k+1}$ are the name of $a$
and a generator not in $\{x_i,x_j\}$. So if subsequent reductions occur that
increase the length of the maximal $i,j$-prefix, then they all involve adjoining the same letter
$a$. (We can't adjoin $a$ and then $a^{-1}$ or the word would not be geodesic!)
So we have $u' =_G u a^r$ with $|u'| = |u| + r$. as claimed. Since
the equation holds for any choice of $w$, we have proved that
$(u')_G$ is the unique longest left divisor of $g$ in $G(i,j)$.

The proof for the maximal right divisor is similar.
\end{proof}

\begin{proposition}\label{ngreduceprop}
\begin{mylist}
\item[(1)]
Let $v,w$ be two geodesic words representing the same group
element $g$, with $\l{v} \ne \l{w}$. Then a single rightward critical sequence
can be applied to $v$ to yield a word ending in $\l{w}$.
\item[(2)] Let $v$ be a freely reduced non-geodesic word with $v=wa$ with
$a \in A$ and $w$ geodesic. Then $v$ admits a rightward length reducing
sequence.
\end{mylist}
\end{proposition}
\begin{proof}
The proof of (1) is by induction on $|v|$, the result being vacuously
true when $|v|=1$.
Let $x_i$ and $x_j$ be the names of  $\l{v}$ and $\l{w}$, and let $v = v'u$,
where  $u$ is the longest 2-generator suffix of $v$. If $u$ has a geodesic
representative ending in $\l{w}$, then the result follows from
\cite[Corollary 2.6]{HoltRees}, so suppose not.
By Proposition \ref{longestld}, we have $\RD_{ij}(g) =_G a^r u$ for some
$a \in \{x_i^{\pm 1},x_j^{\pm 1}\}$, and also $\RD_{ij}(g)$ has a geodesic
representative ending in $\l{w}$.  So, by \cite[Lemma 2.8]{HoltRees},
a single $\tau$-move can be applied to $au$ to give a word ending in $\l{w}$.
Now, by inductive hypothesis, a single rightward critical sequence can be
applied to $v'$ 
to yield a word ending in $a$, and following this by the
above $\tau$-move results in the required critical sequence, which completes
the proof of (1).

If $v$ is as in (2), then $w$ has a geodesic representative ending in $a^{-1}$,
and the result follows from (1).
\end{proof}

In fact the above proposition provides a slightly shorter proof that Artin
groups of large type have FFTP than the one in \cite{HoltRees}.

The next proposition implies that for the groups in Theorem~\ref{main_rd},
for a given $g$, $i$ and $j$,
the letter $a$ arising in Proposition~\ref{longestld} is unique.

\begin{proposition}\label{onetail}
Assume that $G$ satisfies the hypotheses of Theorem~\ref{main_rd},
and let $g \in G$ and $x_i,x_j \in X$ with $i \ne j$.  Then either
\begin{mylist}
\item[(1)] every geodesic word representing $g$ has a prefix that 
represents $LD_{ij}(g)$; or
\item[(2)] there is a unique letter $a \in \{x_i^{\pm 1},x_j^{\pm 1}\}$,
and an element $\LD'_{ij}(g) \in G(i,j)$ such that, for any geodesic
representative $w$ with maximal $\{x_i,x_j\}$-prefix $u$,
we have $\LD_{ij}(g)=_G ua^r$ and $u=_G \LD'_{ij}(g)a^s$,
for $r,s \geq 0$.
\end{mylist} 
In case (1), we define $\LD'_{ij}(g)=\LD_{ij}(g)$, for future reference.
\end{proposition}
\begin{proof}
Let $h := \LD_{ij}(g)$ and $h' := h^{-1}g$,
and suppose that case (1) does not hold.
Then $g$ has at least one geodesic representative whose maximal
$\{x_i,x_j\}$-prefix $u$ is a proper left divisor of $\LD_{ij}(g)$.
Proposition ~\ref{longestld} implies that $h=_Gua^r$ with $r>0$.

If $a$ is not uniquely determined, then $g$ has another geodesic
representative with maximal $\{x_i,x_j\}$-prefix $v$ such that $h=_G vb^s$,
with $s>0$ and $b \neq a$. 
We may assume that the names of $a$ and $b$ are $x_i$, $x_j$ respectively.
Note that, since $h$ has more than one geodesic representative, the vertices
corresponding to $x_i,x_j$ in the graph $\Gamma$ defining $G$ must be
joined by an edge (that is $m_{ij}<\infty$).

Now $a^r h' = u^{-1}g$ has a geodesic representative beginning with $a$ and
another geodesic representative beginning with a letter with name
$x_{i'}$ for some $i' \ne i,j$. By Proposition~\ref{twoposs}, all geodesic
representatives of $a^r h'$ have their longest 2-generator prefix in $G(i,i')$,
and hence all geodesic representatives of $h'$ must begin with the same
letter $c$, of which the name is $x_{i'}$. Furthermore, any geodesic
representative of $\LD_{ii'}(a^r h')$
must contain a critical subword, and similarly $\LD_{ji'}(b^s h')$
contains a critical subword.

From the hypothesis of Theorem~\ref{main_rd} and the fact that the vertices
$x_i$ and $x_j$ are joined by an edge in $\Gamma$, it follows that the edges
joining vertices $x_i,x_{i'}$ and vertices $x_j,x_{i'}$ cannot both have
the label 3. So, for at least one of $\LD_{ii'}(a^r h')$ and
$\LD_{ji'}(b^s h')$, it is true that any geodesic representative has at least
four syllables. Assume without loss that this is true for $\LD_{ii'}(a^r h')$.
So $h'$ has geodesic representatives $u_1,u_2$, where $u_1$ has a
prefix in $G(i,i')$ with at least three syllables and $u_2$ has a prefix
in $G(j,i')$ with at least two syllables. But then, if $c^s$ is the
highest power of $c$ that is a left divisor of $h'$, $c^{-s}h'$ has
a geodesic representative beginning with $x_i$ with a 2-generator prefix
in $G(i,i')$ and another geodesic representative beginning with $x_j$,
which contradicts Proposition~\ref{twoposs} (as applied to prefixes of
$u_1$ and $u_2$).

So $a$ is indeed uniquely determined.
We define $\LD'_{ij}(g) := \LD_{ij}(g)a^{-s}$, where $a^s$
is the maximal power of $a$ that is a right divisor of $\LD_{ij}(g)$.
\end{proof}

The following example shows that Proposition ~\ref{onetail} fails without
the $(3,3,m)$-hypothesis.
Let $m_{12},m_{13},m_{23} = 4,3,3$ and write $a,b,c$ for $x_1,x_2,x_3$.
Let $g$ be the group element represented by the geodesic word
$w := babacabab$. Then, since the only two geodesic representatives
of the suffix $cabab$  of $w$ are $cabab$ and $cbaba$, neither of which
starts with $a^{\pm 1}$ or $b^{\pm 1}$, we see that $\LD_{1,2}(g) = baba$.
But, by replacing $aca$ by $cac$, we have $g =_G w_1$ with
$w_1 = babcacbab$, and also $g =_G (baba)c(abab) =_G aba(bcb)aba =_G w_2$
with $w_2 = abacbcaba$, so $\LD_{1,2}(g) =_G v_1a =_G v_2b$, where
$v_1$ and $v_2$ are the maximal $\{a,b\}$-prefixes of $w_1$ and $w_2$,
respectively. A corresponding example can be constructed with
$m_{12} = m$ for any $m$ with $3 \le m < \infty$.

\subsection{Verifying D1 and D2 for the Artin groups in Theorem~\ref{main_rd}}
We assume for the remainder of the paper that $G$ satisfies the hypotheses
of Theorem~\ref{main_rd}: $G$ is an Artin group of large type that
satsifies the $(3,3,m)$-hypothesis.

We must first  define our set $\PP$
of permissible geodesic factorisations $(g_1,g_2)$.
We use the notation $\PP_{ij}$ to denote the set of permissible geodesic
factorisations in the dihedral Artin group on $G(i,j)$, as defined
in Subsection~\ref{sec:dihedral_D1D2}.
Recall that $\LD_{ij}(g)$ and $\RD_{ij}(g)$ were defined in
Proposition~\ref{longestld} as the longest left and right divisors
of $g$ in $G(i,j)$.
Then we define $\PP$ to be the set of geodesic factorisations $(g_1,g_2)$
in $G$ such that for every pair
of generators $x_i,x_j$, $(\RD_{ij}(g_1),\LD_{ij}(g_2)) \in \PP_{ij}$.

Our proof of D1 will involve an inductive argument. To make this work properly,
we have to investigate in what circumstances we can have a geodesic
factorisation $(ag_1,g_2) \in \PP$ for some $a \in A$, with
$(g_1,g_2) \not\in \PP$. 

\begin{lemma}\label{d1lem1}
Suppose that $g \in G$, $a \in A$, $|ag|=|g|+1$,
$(ag_1,g_2) \in \PP(ag)$ but $(g_1,g_2) \not\in \PP(g)$.
\begin{enumerate}
\item[(1)] Then $ag$ has a geodesic representative that begins with a letter
$b \ne a$.
\item[(2)] 
If $ag_1$ does not lie in any 2-generator subgroup of $G$,
then all geodesic representatives
of $\LD'_{ij}(ag)$ (as defined in Proposition~\ref{onetail})
begin with $b$, where $x_i$ and $x_j$ are the names of $a$ and $b$.
\end{enumerate}
\end{lemma}
\begin{proof}
Suppose first that $ag_1$ lies in a 2-generator subgroup $G(i,j)$.
Then $(ag_1,g_2) \in \PP_{ij}$ and $(g_1,g_2) \not \in \PP_{ij}$
implies that the factorisation $(g_1,g_2)$ of $g$ must be
$\Delta_{ij}$-decreasing (as defined in Section~\ref{sec:dihedral_D1D2}),
while the factorisation $(ag_1,g_2)$ of $ag$ is not.
So $ag_1$ is divisible by $\Delta_{ij}$, and conclusion (1) is immediate.

So now suppose that $ag_1$ does not lie in a 2-generator subgroup of $G$.
Then there exist distinct $i',j'$ such that 
$(\RD_{i'j'}(ag_1),\LD_{i'j'}(g_2)) \in \PP_{i'j'}$ and
$(\RD_{i'j'}(g_1),\LD_{i'j'}(g_2)) \not\in \PP_{i'j'}$.
Let $h := \RD_{i'j'}(ag_1)$.
Since, by Proposition~\ref{longestld},
$\RD_{i'j'}(g_1)$ is a right divisor of $h$, we have $|h| > |\RD_{i'j'}(g_1)|$.
We are assuming that $ag_1 \not\in G(i',j')$, so $|h| < |ag_1|$,
and $ag_1 h^{-1}$ cannot have a geodesic representative beginning with $a$,
as that would imply that $h$ was a right divisor of $g_1$.
Hence all geodesic representatives of $ag_1 h^{-1}$ begin with some letter
$b$ with $b \ne a$, and hence $ag$ has a geodesic representative
beginning with $b$, which proves (1).

We observe also that $(\RD_{i'j'}(g_1),\LD_{i'j'}(g_2))$ must be a
$\Delta_{i'j'}$-decreasing factorisation, whereas $(h,\LD_{i'j'}(g_2))$ is not,
so $h$ must be a signed element divisible by $\Delta_{i'j'}^{\pm 1}$. Hence
$h$ has geodesic representatives beginning with letters with names $x_{i'}$
and $x_{j'}$; these are suffices $u',v'$ of geodesic representatives
$u,v$ of $ag_1$. Now if both $u'$ intersects the maximal $\{x_i,x_j\}$-prefix of
$u$ and $v'$ intersects the maximal $\{x_i,x_j\}$-prefix of $v$, we have
$\{i,j\}=\{i',j'\}$ and hence $ag_1 \in G(i,j)$, a contradiction.
This implies that $ag$ has a
geodesic representative in which the maximal $\{x_i,x_j\}$-prefix is a left
divisor of $ag_1 h^{-1}$.
Then, by Proposition~\ref{onetail}, $\LD'_{ij}(ag)$ is a left divisor of
$ag_1 h^{-1}$; hence, as we showed in the proof of (1), all of its geodesic
representatives begin with $b$, and (2) is proved.

\end{proof}

In Corollary~\ref{geodcor}, we proved D1 for 2-generator Artin groups.
Let $P'$ be a polynomial such that D1 holds with $P'$ for each of the
2-generator subgroups $G(i,j)$ of $G$.

\begin{lemma}\label{d1lem2}
Let $g \in G$ with $|g| = m$ and suppose that 
$v,w$ are geodesic representatives of $g$ with $\f{v} \ne \f{w}$.
Suppose that $0 \le k \le m$. Let $\mathcal{D}$ be the set of
left divisors $g_1$ of $g$ in $\PP_l(g)$ of length $k$
for which all geodesic representatives begin with $\f{v}$.
Then $|\mathcal{D}| \le Q(k):=\sum_{j=1}^kP'(j)$.
\end{lemma}
\begin{proof}
The proof is by induction on $m$, the result being vacuously true for $m=1$.
Let the names of $\f{v}$ and $\f{w}$ be $x_i,x_j$,
let $h:= \LD_{ij}(g)$, and let $h' := \LD'_{ij}(g)$, as defined in
Proposition~\ref{onetail}.
Suppose that $g_1 \in \mathcal{D}$,
with $g=g_1g_2$.

If $g_1 \in G(i,j)$, then $g_1 \in \PP_l(h)$. It follows 
immediately from Corollary~\ref{geodcor} that there
are at most $P'(k)$ such choices for $g_1$. 

So now suppose that $g_1 \not \in G(i,j)$.
Since $g_1$ has $\f{v}$ as unique left divisor of length 1, whereas
$g$ 
and hence $h$, by Proposition \ref{longestld},
has more than one such left divisor, $h$ cannot be a left divisor of $g_1$.

Since $g_1 \not \in G(i,j)$, 
Proposition~\ref{onetail} implies that $h'$ is a left divisor of $g_1$, 
and so a proper divisor of $h$, and that
$h=h'a^r$ for some $r>0$,
$a \in \{x_i^{\pm 1}, x_j^{\pm 1}\}$.
But now, since $\LD_{ij}(g_1)$
is certainly a left divisor of $g$, Proposition~\ref{onetail}
also implies that $h=\LD_{ij}(g_1)a^s$, for some $s>0$. We 
deduce that $\LD_{ij}(g_1)=h'a^t$, with $t\geq 0$.
By~\cite[Lemma 2.8]{HoltRees},
$h' a$ has more than one left divisor of length 1,
and so cannot be a left divisor of $g_1$, so $t=0$.
Hence $\LD_{ij}(g)=h'$.

Now $h'^{-1}g$ has a geodesic representative with first letter $a$,
whereas $g_1' := h'^{-1}g_1$ does not. So by Proposition~\ref{twoposs}, all
geodesic representatives of $g_1'$ have the same first letter.

We claim that $g_1' \in \PP_l(h'^{-1} g)$. If not then for some $i',j'$,
$(g'_1,g_2)$ is $\Delta_{i'j'}$-decreasing whereas
$(g_1,g_2)$ is not, and hence $\RD_{i'j'}(g_1)$ is longer than
$\RD_{i'j'}(g'_1)$,
Now suppose that $vw$ is a geodesic representative of $g_1$, with $w$
representing $\RD_{i'j'}(g_1)$. Then by
Proposition~\ref{onetail} (2) some prefix $u$ of $vw$ represents $h'$. 
If $u$ is a prefix of $v$, then $w$ is a suffix of a representative of $g'_1$,
and $\RD_{i'j'}(g_1)$ is a right divisor of $g'_1$, yielding a contradiction.
So $u$ and $w$ must intersect non-trivially, and hence $|\{i,j\} \cap
\{i',j'\}| = 1$. We may suppose that $i=i'$ and $j \ne j'$.
Then since $(g_1,g_2)$ is not $\Delta_{i'j'}$-decreasing, 
the element $\RD_{i'j'}(g_1)$ is divisible by $\Delta_{i'j'}$, and so has a
representative that begins with $x_{j'}^{\pm 1}$; hence
we can find a geodesic representative of $g_1$, and then one of $g$, in
which the longest $\{x_i,x_j\}$-prefix is shorter than $|h'|$, contradicting
Proposition~\ref{onetail} (2). 

Now by the inductive hypothesis applied to $h'^{-1} g$, there are at most
$Q(k-|h'|) \leq Q(k-1)$ possible $g_1'$ and hence at most 
$Q(k-1)$ possible $g_1$ of length $k$ not in $G(i,j)$.  
This bounds the total number of $g_1$ of length $k$
by $Q(k-1)+P'(k)=Q(k)$ for any $k \leq m$, and so completes the inductive step.  
\end{proof}

\begin{proposition}
D1 holds with the polynomial $P_1$ defined by $P_1(n) = 2 \sum_{j=0}^n Q(j)$,
where the set $\PP$ and the polynomial $Q$ are as defined above.
\end{proposition}
 
\begin{proof} Let $g$ be an element of length $m$ and choose $k$ with
$0 \le k \le m$. The proof is by induction on $m$, the case $m=1$ being clear.

If all geodesic representatives of $g$ begin with the same letter, then
the result follows immediately from the inductive hypothesis and
Lemma \ref{d1lem1} (1). So suppose that $g$ has geodesic representatives
$v,w$ with $\f{v} \ne \f{w}$, and let $x_i,x_j$ be the names of
$\f{v},\f{w}$. Assume without loss that $\LD'_{ij}(g)$ has a geodesic
representative beginning with $\f{w}$.

Let $\D$ be the set of elements in $\PP_l(g)$ of length $k$. 
Then $\D = \D_1 \cup \D_2 \cup \D_3$, where $\D_1$ consists of those elements of
$\D$ that lie in $G(i,j)$, $\D_2$ consists of those elements of 
$D \setminus \D_1$ with at least one geodesic representative beginning with
$\f{w}$, and $\D_3$ consists of those elements of $\D \setminus \D_1$ all of
whose geodesic representatives begin with $\f{v}$.
Then $|\D_1| \le Q(k)$ from the 2-generator case
and $|\D_3| \le Q(k)$ by Lemma~\ref{d1lem2}.

Suppose that $g_1 \in \D_2$. If
$\f{w}^{-1} g_1 \not\in \PP_l(\f{w}^{-1} g)$, then by Lemma~\ref{d1lem1} (2),
all geodesic representatives
of $\LD'_{ij}(g)$ begin with $\f{v}$, contrary to assumption; so
$\f{w}^{-1} g_1 \in \PP_l(\f{w}^{-1} g)$, and 
the inductive hypothesis applied to $\f{w}^{-1} g_1$ gives
$|\D_2| \le P_1(k-1)$. 
So $|\D| \le P_1(k-1)+2Q(k)=P_1(k)$ as required.
\end{proof}

The following lemma is of the same type as Lemma~\ref{nclemma1}.
It holds for all Artin groups of large type.

\begin{lemma}\label{nclemma2}
For some $i \ne j$, suppose that no nontrivial element of
$G(i,j)$  is either a left divisor of $g_1$ or a right divisor of $g_2$.
If $g_1g_2 \in G(i,j)$, then $g_1g_2 = 1$.
\end{lemma}
\begin{proof} Suppose $g_1g_2 = h$ with $1 \ne h \in G(i,j)$.
Since $g_1$ has no element of $G(i,j)$ as left divisor, $h g_2^{-1}$ cannot
be a geodesic factorisation. Let $u, v$ be geodesic words representing
$h, g_2^{-1}$. Let $u' = a u''$ with $a \in \{x_i^{\pm 1},x_j^{\pm 1} \}$ be the
shortest suffix of $u$ such that $u'v$ is non-geodesic.  Then
$u''v$ has a geodesic representative beginning with $a^{-1}$ whereas
$u''$ does not.
So the maximal $\{x_i,x_j\}$-prefix of that geodesic representative
of $u''v$ cannot represent a left divisor of $(u'')_G$, and hence
by Lemma~\ref{longestld}
$(u'')_G \neq \LD_{ij}((u''v)_G)$.
But then $g_2$ must have
a nontrivial right divisor in $G(i,j)$, contrary to assumption.
\end{proof}

It remains to verify D2.
We start by defining a merging process for elements $g_1 \in C_k$ and
$g_2 \in C_l$ which, as in the 2-generator case, will result in a merger
$(f_1,\Delta_{ij}^r,f_2)$ for some $i,j$ where, for $h_1 := f_1^{-1} g_1$,
$h_2 := g_2 f_2^{-1}$, we have $h_1h_2 = \Delta_{ij}^r$,
$(f_1,h_1) \in \PP(g_1)$ and $(h_2,f_2) \in \PP(g_2)$.

As in the 2-generator case, the merging process proceeds through a series
of steps, starting with  $g_1^{(1)}:=g_1$, $r_1:=0$, $g_2^{(1)}:=g_2$,
and in the $t$-th step we compute $g_1^{(t+1)},r_{t+1},g_2^{(t+1)}$,
by choosing geodesic factorisations
$g_1^{(t+1)}h$ and $h' g_2^{(t+1)}$ of $g_1^{(t)}$ and $g_2^{(t)}$
such that $g_1^{(t+1)} \in \PP_l(g_1)$ and $g_2^{(t+1)} \in \PP_r(g_2)$.
But there is an additional complication in that the $i,j$ in the
term $\Delta_{ij}^{r_t}$ may change during the process.

Recall from Subsection~\ref{sec:dihedral_D1D2} that there are three types
of merging steps, the first being free cancellation, and the other two
involving a specific $\Delta_{ij}$. In general, they are defined in the same way
as in the 2-generator case subject to the condition that, if $r_t \ne 0$, then
we must choose $h,h' \in G(i,j)$. If $r_t = 0$, then there is no such
restriction on free cancellation moves whereas, for the other two types,
we have $h,h' \in G(i,j)$ for the new values of $i,j$.

Again as in the 2-generator case, for $g \in G$, $k,l \in \N$ $k+l=m:=|g|$ we
define the set $S(g,k,l)$ to be the set of all triples
$(f_1,\Delta^r,f_2)$ that arise as mergers of pairs of elements $g_1,g_2$
of lengths $k,l$ and with $g_1g_2= g$.
For each such triple,
we have $\Delta = \Delta_{ij}$ for some $i,j$; if $r=0$, then $\{i,j\}$ is not
necessarily uniquely defined. 
There exist $h_1,h_2 \in G$ with
$h_1h_2 = \Delta^r$ where $g_1=f_1h_1$ and $g_2=h_2f_2$, and the restrictions
on the merging steps described above ensure that
$(f_1,h_1) \in \PP(g_1)$ and $(h_2,f_2) \in \PP(g_2)$.

In order to verify D2, 
we have to show that $|S(g,k,l)|$ is bounded by $P(k)$ for some polynomial $P$.
We decompose $S$ as a disjoint union $S_0 \cup S_1 \cup S_2$, and 
establish polynomial bounds for each of those subsets.

We define $S_0$ to be the subset of elements of $S(g,k,l)$ for which
$r=0$ and $f_1,f_2 \in \langle x_i \rangle$ for some generator $x_i$.
Since $|f_1| \le k$, $f_2 = f_1^{-1}g$, and $x_i$ is determined by $g$,
we have $|S_0| \le 2k+1$. 

The sets $S_1,S_2$ are defined with respect to the elements
$f''_1,\hat{f},f''_2$ that are defined in Proposition~\ref{reduceto2genprop}
below.  We define $S_1$ to be the set of triples in
$S(g,k,l)\setminus S_0$ for which $f''_1\hat{f}f''_2$ is a geodesic
factorisation of $g$, and $S_2$ to be the others.

\begin{proposition}\label{reduceto2genprop}
Let $(f_1, \Delta^r, f_2) \in S(g,k,l) \setminus S_0$.
Then the following is true for some $i \ne j$ with $1 \le i,j \le n$ where,
if $r \ne 0$, $i,j$ are determined by $\Delta = \Delta_{ij}$.
There exist $f''_1, f'_1, f''_2, f'_2, \hat{f} \in G$ such that
\begin{mylist}
\item[(1)] $f'_1,f'_2,\hat{f} \in G(i,j)$, but $f''_1$ has no right divisor in
$G(i,j)$ and $f''_2$ has no left divisor in $G(i,j)$,
\item[(2)] $g = f''_1 \hat{f} f''_2$,
$f_1 = f''_1 f'_1$, $f_2 = f'_2 f''_2$, and
\item[(3)] $(f'_1,\Delta^r,f'_2) \in
S(\hat{f},k',l')$ for some $k' \le k$ and $l' \le l$.
\end{mylist}

Furthermore, $\hat{f}$ does not lie in $\langle x_i \rangle$ for any $i$.
\end{proposition}

\begin{proof}\
Since $(f_1, \Delta^r, f_2) \in S(g,k,l)$, it is the merger of
elements $g_1 \in C_k$ and $g_2 \in C_l$, and so there exist
$h_1,h_2 \in G$ with $g_1 = f_1h_1$, $g_2 = h_2f_2$ and $h_1h_2 = \Delta^r$.

Suppose first that $r \ne 0$, and hence that $i,j$ are uniquely specified,
with $m_{ij}< \infty$. 
Define $f'_1:= \RD_{ij}(f_1)$, $h'_1:= \LD_{ij}(h_1)$,
$f'_2:= \LD_{ij}(f_2)$, $h'_2:= \RD_{ij}(h_2)$, $k'=|f'_1h'_1|$,
$l'=|h'_2f'_2|$ and $\hat{f}=f'_1h'_1h'_2f'_2$.
Now Lemma~\ref{nclemma2} applied to $h_1'^{-1}h_1$ and $h_2h_2'^{-1}$
tells us that $h'_1h'_2=h_1h_2 = \Delta^r$.
In other words, during the merging the right divisor 
$(h'_1)^{-1}h_1$ of $h_1$ has cancelled with
the left divisor $h_2(h'_2)^{-1}$ of $h_2$.
We see also that $(f'_1,\Delta^r,f'_2) \in S(\hat{f},k',l')$
is the result of the merging of $f'_1h'_1$ and $h'_2f'_2$ in $G(i,j)$.
Suppose that $\hat{f} \in \langle x_{i'} \rangle$ for $i'=i$ or $j$.
Then the application of Lemma~\ref{nclemma1} to $f'_1h'_1$ and $h'_2f'_2$
shows that free cancellation of $w$ with
$w \in \PP_r(f'_1h'_1)$, $w^{-1} \in \PP_l(h'_2f'_2)$ can be used to
merge  $f'_1h'_1$ and $h'_2f'_2$ to give just a power
of $x_{i'}$, contradicting the fact that $r \neq 0$.
(Here we are using the factorisations in $\PP^1$,
to ensure that such a cancellation would be permissible.)

Suppose instead that $r=0$.  Then $h_1=h_2^{-1}$.
Since we are assuming that $(f_1, \Delta^r, f_2) \not\in S_0$,
we can choose distinct $i,j$ such that $f'_1 := \RD_{ij}(f_1)$ and
$f'_2:= \LD_{ij}(f_2)$ do not both lie in the same cyclic subgroup
$\langle x_{i'} \rangle$ for ${i'}=i$ or $j$.
If possible, we choose $i,j$ such that the vertices corresponding to
$i,j$ in the graph $\Gamma$ defining $G$ are joined by an edge (equivalently,
$m_{ij}< \infty$).  Again we define $h'_1:= \LD_{ij}(h_1)$, $h'_2 :=
\RD_{ij}(h_2)$, $k'=|f'_1h'_1|$, $l'=|h'_2f'_2|$ and $\hat{f}=f'_1h'_1h'_2f'_2$.
Since $h_1=h_2^{-1}$, we have $h'_1= (h'_2)^{-1}$,
so in this case too $(f'_1,\Delta_{ij}^r,f'_2) \in
S(\hat{f},k',l')$ is a merger of $f'_1h'_1$ and $h'_2f'_2$ in $G(i,j)$ and,
by Lemma~\ref{nclemma1}, $f'_1h'_1h'_2f'_2$ does not
lie in a cyclic subgroup  $\langle x_{i'} \rangle$ for $i'=i$ or $j$.
\end{proof}

For triples in $S_1$, we claim that $f''_1 \in \PP_l(g)$. Since $f''_1$
has no right divisor in $G(i,j)$ the only way that the factorisation
$(f''_1,\hat{f}f''_2)$ could be $\Delta$-decreasing for some $\Delta$
would be with $\Delta=\Delta_{ii'}$, where $x_{i'}$ is the name of
both $\l{f''_1}$ and $\f{f''_2}$ and $\hat{f} = x_i^sx_j^t$ for some $s,t>0$.
But then
$\RD_{ii'}(\hat{f}f''_2)$ would have at most two syllables, and so
the factorisation would still be permissible. So the claim is true.

Hence, since $|f''_1| \le k$, the number of choices of $f''_1$ is bounded by
$P_1(k)$.  For given $f''_1$ and a pair $i,j$, we have
$\hat{f} = LD_{ij}({f''_1}^{-1}g)$ and $f''_2 = \hat{f}^{-1}{f''_1}^{-1}g$.
Now $(f'_1,\Delta_{ij}^r,f'_2) \in S(\hat{f},|f'_1h'_1|,|h'_2f'_2|)$,
which we know from the dihedral case has size bounded by
$P^{ij}_2(\min{|f'_1h'_1|,|h'_2f'_2)}) \leq P^{ij}_2(k)$.
So $|S_1| \le P_1(k) \sum_{i,j}P^{ij}_2(k)$.

In order to bound $|S_2|$, we need some further technical lemmas.

\begin{lemma}\label{geodfaclem1}
Suppose that $f''_1\hat{f}f''_2$ is not a geodesic factorisation of $g$.
Then $\hat{f}=a^sb^t$
for some $s,t>0$ where $a,b\in  \{ x_i^{\pm 1},x_j^{\pm 1} \}$ and
$a,b$ have different names.
\end{lemma}
\begin{proof}
Let $u''_1, \hat{u}, u''_2$ be geodesic words representing $f''_1,\hat{f},
f''_2$, respectively. Suppose (for a contradiction) that $\hat{u}$ has at
least three syllables. Since $u''_1\hat{u} u''_2$ is not geodesic, we can
by Proposition \ref{ngreduceprop} (2) apply a rightward length reducing
sequence to a prefix of it containing $u''_1$.

If the left-hand end of the
first critical subword in the sequence were to the left of $\hat{u}$, then some
$\tau$-move in the sequence would
replace a letter in the first syllable of
$\hat{u}$ by a letter $c$ with name not equal to $x_i$ or $x_j$. But since
$c$ would then be followed by at least two syllables with names $x_i$ and
$x_j$, $c$ 
would not be the leftmost letter in a critical subword, and so this would be
the last $\tau$-move in the sequence, and would not provoke a free reduction.  
So the left hand end of the first
critical subword must be  within $\hat{u}$.

We may assume that the first $\tau$-move of the sequence is not completely
within $\hat{u}$, since otherwise we could just
replace $\hat{u}$ by the result of this move. So the first critical subword
overlaps the right hand end of $\hat{u}$ and hence has its left hand end in
the final syllable, $a^s$ say, of $\hat{u}$. But then $a^su''_2$ is not
geodesic and so, by Corollary \ref{twoposscor}, neither is $au''_2$.
But then $f''_2$ has a geodesic representative beginning with $a^{-1}$,
contradicting $f_2 = f'_2 f''_2$ with $f'_2 = \LD_{ij} (f_2)$. 
\end{proof}

\begin{lemma}\label{geodfaclem2}
Suppose that $\hat{f} = a^s b^t$ for some $s,t > 0$, where
$a,b$ have distinct names $x_i,x_j$,
and that $f''_1\hat{f}f''_2$ is not a geodesic factorisation of $g$.
Then there exists a letter $c$ with name $x_{i'}$ not equal to $x_i$ or $x_j$,
$q>0$, and $e_1,e_2 \in G$, such that  $f_1''a^s = e_1 c^q$ and
$b^t f_2'' = c^{-q} e_2$,
where $e_1c^q$, $c^{-q}e_2$ and  $e_1 e_2$ are all geodesic factorisations.
Furthermore, the vertices corresponding to $x_i,x_j$ in the graph $\Gamma$
are joined by an edge; that is $m_{ij}< \infty$.

\end{lemma}
\begin{proof}
Let $u''_1,u''_2$ be geodesic representatives of $f''_1,f''_2$.
Using similar reasoning to that of the proof of the previous lemma, we see
that a rightward length reducing sequence applied to
the word $u''_1 a^s b^t u''_2$ would consist of a sequence of moves that
replaced $u''_1 a^s$ by a word ending in a letter $c$ with name $x_{i'}$ not
equal to $x_i$ or $x_j$, followed by a length reducing sequence applied to
$c b^t u''_2$. But then $b^t u''_2$ would have a geodesic representative
starting with $c^{-1}$. Let $q$ be maximal such that $f''_1 a^s$ and
$b^t f''_2$ have geodesic factorisations $e_1 c^q$ and $c^{-q} e_2$
respectively.

If $m_{ij}$ were infinite, then we would have $r=0$ in
Proposition~\ref{reduceto2genprop}, in which case we could and would have
chosen one of the pairs $i,i'$ or $j,i'$, depending on which of $m_{ii'}$ and
$m_{ji'}$ was finite rather than $i,j$
in the proof of that proposition. 
So $m_{ij}< \infty$, which proves the final assertion
in the lemma. Hence the hypothesis of Theorem~\ref{main_rd} implies that at
least one of the $m_{ii'}$ and $m_{ji'}$ is at least 4.

Let $h_1$ and $h_2$ be the group elements represented by
$e_1 c^q =_G u''_1 a^s$ and $c^{-q} e_2 =_G b^t u''_2$, respectively.
Applying Propositions \ref{twoposs} and \ref{longestld} to these elements,
we see that $\RD_{ii'}(h_1)$ and $\LD_{ji'}(h_2)$ are represented by
$w_1$ and $w_2$, ending in $c^q$ and beginning with $c^{-q}$, respectively.
Since $w_1$  and $w_2$ are not the unique geodesic representatives
of $\RD_{ii'}(h)$ and $\LD_{ji'}(h_2)$, the above condition on the 
edge labels implies that at least one of $w_1$ and $w_2$ has at least
four syllables.  Hence either $e_1$ has a geodesic representative
$v_1$ in which at least the final three syllables are powers
of $a$ and $c$, or $e_2$ has a geodesic representative $v_2$ in
which at least the first three syllables are powers $b$ and $c$.
We see that
no rightward length reducing sequence can pass through this middle section,
and hence $v_1 v_2$ is geodesic, which completes the proof.
\end{proof}

By Lemma~\ref{geodfaclem1}, each element $(f_1,\Delta_{ij}^r,f_2)$ of $S_2$
determines $f'',\hat{f},f''_2$ for which
the hypotheses and conclusion of Lemma~\ref{geodfaclem2} hold.
Since $|g| \ge l-k$, we must have $q \le k$, so there are
at most $2nk$ possibilities for $c^q$. So assume that $c^q$ is fixed.
Unfortunately we have no bound on $|e_1|$ as a polynomial in $k$.  Let
$e_1 = e''_1 e'_1$ with $e'_1 = \RD_{ii'}(e_1)$. Then $|e''_1| < |f''_1| < k$
and $e''_1 \in \PP_l(e_1)$, so there are at most $P_1(k)$ possible $e''_1$.

We claim that  the triple $(f_1,\Delta_{ij}^r,f_2)$ is uniquely
determined by $g$, $c^q$ and $e''_1$.
To see this note that $\LD_{ii'}({e_1''}^{-1} g) = e_1 c^{-t}$ for some
$t \ge 0$ and, since $(e_1,c^q)$ is a geodesic factorisation,
$e_1$ does not have $c^{-1}$ as right divisor. Hence, given
$g$, $c^q$ and $e''_1$, $e'_1$ is determined as
$\LD_{ii'}({e_1''}^{-1} g) c^t$, where $c^{-t}$ is the longest right
divisor of $\LD_{ii'}({e_1''}^{-1} g)$ of that form.
Then $e_1$ and $f''_1$ are determined as $e''_1 e'_1$ and $e_1 c^q a^{-s}$,
where $a^s$ is the highest power of $a$ that is a right divisor of $e_1c^q$.
Then, as in the previous case, we have $\hat{f} = LD_{ij}({f_1''}^{-1}g)$ and
$f''_2 = \hat{f}^{-1}{f_1''}^{-1}g$, which establishes the claim.

So the number of triples $(f_1,\Delta_{ij}^r,f_2) \in S(g,k,l)$ that satisfy
the conclusion of Lemma~\ref{geodfaclem2}, and hence also $|S_2|$, is bounded
by $2nkP_1(k)\sum_{i,j} P^{ij}_2(k)$. 

This completes the proof that Artin groups satisfying the hypotheses of
Theorem~\ref{main_rd} satisfy D1 and D2,
and hence also the proof of Theorem~\ref{main_rd}.

\section*{Acknowledgments}
We would like to thank the referee of an earlier version of this paper, whose
query as to whether or not Lafforgue's results would actually apply
to the groups under study led us to address this question and hence add the
second of the two main results.

The first author was partially supported by the Marie Curie Reintegration
Grant 230889, the third author by grants from the  Mathematics Departments
of the Universities of Fribourg and Neuch\^atel, Switzerland, and
by UK Engineering and Physical Sciences Research Council grant EP/F014945/1,
Quantum Computation: Foundations,
Security, Cryptography and Group Theory.

\noindent Addresses:

\medskip
\noindent Laura Ciobanu,\\
Mathematics Department,
University of Neuch\^atel,
Rue Emile Argand 11,\\
CH-2000 Neuch\^atel, Switzerland.\\
email: Laura.Ciobanu@unine.ch

\medskip
\noindent Derek F. Holt,\\
Mathematics Institute,
University of Warwick,
Coventry CV4 7AL,
UK.\\
email: D.F.Holt@warwick.ac.uk

\medskip
\noindent Sarah Rees,\\
School of Mathematics and Statistics,
University of Newcastle,
Newcastle\\ NE1 7RU,
UK.\\
email: Sarah.Rees@newcastle.ac.uk
\end{document}